\documentclass[12pt]{article}

\usepackage{amsmath,amssymb,amsthm}
\usepackage[margin=1in]{geometry}
\usepackage{hyperref}
\usepackage{tikz}
\usepackage{graphicx}
\usepackage{url}
\usepackage{enumitem}
\usepackage[mathlines]{lineno}
\usepackage{multirow}
\usepackage{dsfont} % needed for double-struck 1
\usepackage{color}
\usepackage{multicol}
\usepackage{longtable}
\usepackage[title]{appendix}
\definecolor{red}{rgb}{1,0,0}

\newtheorem{thm}{Theorem}[section]
\newtheorem{cor}[thm]{Corollary}
\newtheorem{lem}[thm]{Lemma}
\newtheorem{prop}[thm]{Proposition}

\newtheorem{obs}[thm]{Observation}

\newtheorem{quest}[thm]{Question}
\newtheorem{rem}[thm]{Remark}
\theoremstyle{definition}
\newtheorem{defn}[thm]{Definition}
\theoremstyle{example}

\DeclareMathOperator{\epr}{epr}
\DeclareMathOperator{\sepr}{sepr}
\DeclareMathOperator{\rank}{rank}
\DeclareMathOperator{\diag}{diag}

\DeclareMathOperator{\uepr}{uepr}
\newcommand{\R}{\mathbb{R}}
\newcommand{\C}{\mathbb{C}}

\newcommand{\Rnn}{\R^{n\times n}}

\newcommand{\Cnn}{\C^{n\times n}}

\newcommand{\F}{\mathcal{F}}
\newcommand{\T}{\mathcal{T}}

\newcommand{\negat}{\operatorname{neg}}
\newcommand{\OL}{\overline}
\newcommand{\VierOneA}{R_{\tt A^+A^*A^*A^+}}
\newcommand{\VierOneB}{R_{\tt A^+A^+A^*A^-}}
\newcommand{\VierOneC}{R_{\tt A^+A^-A^*A^+}}
\newcommand{\VierOneD}{R_{\tt A^*A^*A^+N}}
\newcommand{\VierOneE}{R_{\tt A^+A^*S^-N}}
\newcommand{\VierOneF}{R_{\tt A^*S^*A^*A^+}}
\newcommand{\VierOneG}{R_{\tt A^+S^*A^*A^-}}
\newcommand{\VierOneH}{R_{\tt A^+S^+A^*A^-}}
\newcommand{\VierOneI}{R_{\tt A^+S^+A^-A^-}}
\newcommand{\VierOneJ}{R_{\tt A^+S^-A^*A^+}}
\newcommand{\VierTwoA}{R_{\tt A^*S^*A^+N}}
\newcommand{\VierTwoB}{R_{\tt A^*S^-A^+N}}
\newcommand{\VierTwoC}{R_{\tt A^-S^+A^+N}}
\newcommand{\VierThreeA}{R_{\tt NA^-A^*A^+}}
\newcommand{\VierThreeB}{R_{\tt NS^-S^*A^+}}
\newcommand{\VierThreeC}{R_{\tt S^+A^*A^*A^-}}
\newcommand{\VierFourA}{R_{\tt S^*A^*A^*A^+}}
\newcommand{\VierFourB}{R_{\tt S^*A^-A^+A^+}}
\newcommand{\VierFourC}{R_{\tt S^*A^-A^+A^-}}
\newcommand{\VierFourD}{R_{\tt S^+A^*A^*A^+}}
\newcommand{\VierFourE}{R_{\tt S^+A^-A^+A^+}}
\newcommand{\VierFourF}{R_{\tt S^+A^-A^+A^-}}
\newcommand{\VierFourG}{R_{\tt S^+A^-A^-A^+}}
\newcommand{\VierFourH}{R_{\tt S^+A^-A^-A^-}}
\newcommand{\VierFourI}{R_{\tt S^*A^*S^*A^+}}
\newcommand{\VierFourJ}{R_{\tt S^+A^*S^-A^+}}
\newcommand{\VierFourK}{R_{\tt S^*S^*A^*A^+}}
\newcommand{\VierFourL}{R_{\tt S^*S^-A^+A^+}}
\newcommand{\VierFourM}{R_{\tt S^*S^-A^+A^-}}
\newcommand{\VierFourN}{R_{\tt S^+S^*A^*A^+}}
\newcommand{\VierFourO}{R_{\tt S^+S^-A^*A^+}}
\newcommand{\VierFiveA}{R_{\tt S^*A^*A^+A^+}}
\newcommand{\VierFiveB}{R_{\tt S^*A^*A^+A^-}}
\newcommand{\VierFiveC}{R_{\tt S^*A^*S^+A^+}}
\newcommand{\VierFiveD}{R_{\tt S^*S^*A^+A^+}}
\newcommand{\VierSixA}{R_{\tt S^*S^*S^*A^+}}
\newcommand{\VierSixB}{R_{\tt S^*S^*S^+A^-}}
\newcommand{\FiveOneA}{R_{\tt A^*A^*A^+A^*A^-}}
\newcommand{\FiveOneB}{R_{\tt A^*A^*S^+A^*A^-}}
\newcommand{\FiveOneC}{R_{\tt A^*S^*A^+A^*A^-}}
\newcommand{\FiveOneD}{R_{\tt A^*S^-A^+A^*A^-}}
\newcommand{\FiveOneE}{R_{\tt A^*S^*S^+A^*A^-}}
\newcommand{\FiveOneF}{R_{\tt A^*S^-S^+A^*A^-}}
\newcommand{\FiveTwoA}{R_{\tt A^*A^*A^+S^*A^-}}
\newcommand{\FiveTwoB}{R_{\tt A^*A^*A^+S^+A^-}}
\newcommand{\FiveTwoC}{R_{\tt A^*A^*A^+S^-A^-}}
\newcommand{\FiveTwoD}{R_{\tt A^*A^*S^+S^*A^-}}
\newcommand{\FiveTwoE}{R_{\tt A^*A^*S^+S^+A^-}}
\newcommand{\FiveTwoF}{R_{\tt A^*A^*S^+S^-A^-}}
\newcommand{\FiveTwoG}{R_{\tt A^*S^*A^+S^*A^-}}
\newcommand{\FiveTwoH}{R_{\tt A^*S^*A^+S^+A^-}}
\newcommand{\FiveTwoI}{R_{\tt A^*S^*A^+S^-A^-}}
\newcommand{\FiveTwoJ}{R_{\tt A^*S^-A^+S^*A^-}}
\newcommand{\FiveTwoK}{R_{\tt A^*S^-A^+S^+A^-}}
\newcommand{\FiveTwoL}{R_{\tt A^*S^-A^+S^-A^-}}
\newcommand{\FiveTwoM}{R_{\tt A^*S^*S^+S^*A^-}}
\newcommand{\FiveTwoN}{R_{\tt S^*S^*S^+S^+A^-}}
\newcommand{\FiveTwoO}{R_{\tt S^*S^*S^+S^-A^-}}
\newcommand{\FiveTwoP}{R_{\tt S^*S^-S^+S^*A^-}}
\newcommand{\SixOneA}{R_{\tt A^-A^*S^+A^+A^*A^-}}
\newcommand{\SixOneB}{R_{\tt A^-A^*S^+A^+S^*A^-}}
\newcommand{\SixOneC}{R_{\tt A^-A^*S^+S^+A^*A^-}}
\newcommand{\SixOneD}{R_{\tt A^-A^*S^+S^+S^*A^-}}
\newcommand{\NSFrealA}{R_{\tt A^+S^+A^-A^+}}
\newcommand{\NSFrealB}{R_{\tt A^-NS^+NA^-}}
\newcommand{\NSFrealC}{R_{\tt A^-S^+A^+S^-A^-}}
\newcommand{\NSFrealD}{R_{\tt A^-S^*S^+A^+S^+A^-}}
\newcommand{\NSFrealE}{R_{\tt NS^-NS^+S^*A^-}}
\newcommand{\NSFrealF}{R_{\tt NS^-NS^+S^+A^-}}
\newcommand{\ComOneA}{C_{\tt NS^-NA^+S^*A^-}}
\newcommand{\ComOneB}{C_{\tt NS^-NA^+S^+A^-}}
\newcommand{\ComTwoA}{C_{\tt NA^-S^*A^+}}
\newcommand{\ComTwoB}{C_{\tt NA^-S^+A^+}}
\newcommand{\NSFcomA}{C_{\tt A^-NA^+A^*A^-}}
\newcommand{\NSFcomB}{C_{\tt NA^-NA^+N}}
\begin{document}
\title{On the signs of the principal minors of Hermitian matrices}
\author{Xavier Mart\'inez-Rivera\thanks{Department of Mathematics, Bates College, Lewiston, ME 04240, USA (martinez.rivera.xavier@gmail.com).
}
\and
Kamonchanok Saejeam\thanks{Department of Physics and Astronomy, Bates College, Lewiston, ME 04240, USA (ksaejeam@bates.edu).
}
}

%\linenumbers

\maketitle

\begin{abstract}
The signed enhanced principal rank characteristic sequence (sepr-sequence) of a given $n \times n$ Hermitian matrix $B$ is the sequence 
$t_1t_2 \cdots t_n$, where
$t_k$ is $\tt A^*$, $\tt A^+$, $\tt A^-$, $\tt N$, $\tt S^*$, $\tt S^+$, or $\tt S^-$, based on the following criteria:
%%%%%%%%%%%%%%
%%%%%%%%%%%%%%
$t_k = \tt A^*$ if all the order-$k$ principal minors of $B$ are nonzero, and two of those minors are of opposite sign;  
%%%%%%%%%%%%%%
%%%%%%%%%%%%%%
$t_k = \tt A^+$ (respectively, $t_k = \tt A^-$) if all the order-$k$ principal minors of $B$ are positive (respectively, negative);
%%%%%%%%%%%%%%
%%%%%%%%%%%%%%
$t_k = \tt N$ if all the order-$k$ principal minors of $B$ are zero;
%%%%%%%%%%%%%%
%%%%%%%%%%%%%%
$t_k = \tt S^*$ if $B$ has a positive, a negative, and a zero order-$k$ principal minor;
%%%%%%%%%%%%%%
%%%%%%%%%%%%%%
$t_k = \tt S^+$ (respectively, $t_k = \tt S^-$) if $B$ has both a zero and a nonzero order-$k$ principal minor, and all the nonzero order-$k$ principal minors of $B$ are positive (respectively, negative).
%%%%%%%%%%%%%%
%%%%%%%%%%%%%%
A complete characterization of the sequences of order $2$ and order $3$
that do not occur as a subsequence of the sepr-sequence of any Hermitian matrix is presented
(a sequence has order $k$ if it has $k$ terms).
An analogous characterization for real symmetric matrices is presented as well.
\end{abstract}


\medskip


%%%%%%%%%%%%%%%%%%%%
%%%%%%%%%%%%%%%%%%%%
%%%%%%%%%%%%%%%%%%%%
%%%%%%%%%%%%%%%%%%%%
%%%%%%%%%%%%%%%%%%%%
%%%%%%%%%%%%%%%%%%%%
%%%%%%%%%%%%%%%%%%%%
%%%%%%%%%%%%%%%%%%%%
%%%%%%%%%%%%%%%%%%%%
\section{Introduction}\label{s: intro}
$\null$
\indent
Here, we continue the investigation of the signs of the principal minors of Hermitian matrices that was initiated by the first author in \cite{XMR-sepr}. At the center of the investigation is the sepr-sequence, whose introduction, in \cite{XMR-sepr}, was motivated by the work of Brualdi et al.\  \cite{pr}, who introduced the pr-sequence, as well as by the work of Butler et al.\ \cite{EPR}, who introduced the epr-sequence.  Part of the motivation for the introduction of the pr-, epr- and sepr-sequences stems from the \textit{principal minor assignment problem}, which is stated in \cite{HS}.  
Additional sources of motivation are the instances in which the principal minors of a matrix are of interest, examples of which were stated in \cite[p.\ 1485]{XMR-sepr}.

The definition of certain terminology is needed in order to state the definition of  `sepr-sequence,' including the definition of `epr-sequence.' (The definition of `pr-sequence' is omitted, as no other reference to pr-sequences is made in this paper.)
For a given positive integer $n$, 
$[n]:=\{1,2, \dots,n\}$.
A given matrix $B \in \Cnn$ is 
said to have {\em order} $n$;
if $\alpha, \beta \subseteq \{1,2, \dots,n\}$, then
$B[\alpha, \beta]$ denotes the submatrix of $B$ lying in 
the rows that are indexed by $\alpha$ and 
the columns that are indexed by $\beta$;
we refer to $B[\alpha]:=B[\alpha, \alpha]$ as a \textit{principal} submatrix;
the determinant of a 
$k \times k$ principal submatrix of $B$ is a
\textit{principal} minor of $B$, 
and such a minor has \textit{order} $k$.
For a given Hermitian matrix $B \in \Cnn$,
the \textit{enhanced principal rank characteristic sequence} (\textit{epr-sequence}) of $B$,
which is denoted by $\epr(B)$, was defined in \cite{EPR} to be
$\epr(B) = \ell_1\ell_2 \cdots \ell_n$, where
   \begin{equation*}
      \ell_k =
         \begin{cases}
             \tt{A} &\text{if all the principal minors of $B$ of order $k$ are nonzero;}\\
             \tt{S} &\text{if some but not all the principal minors of $B$ of order $k$ are nonzero;}\\
             \tt{N} &\text{if none of the principal minors of $B$ of order $k$ are nonzero.}
         \end{cases}
   \end{equation*}
We now have the necessary terminology to state the definition of  `sepr-sequence.'
\begin{defn}\rm{\cite[Definition 1.1]{XMR-sepr}}
\normalfont
Let $B \in \Cnn$ be a Hermitian matrix and 
$\epr(B) = \ell_1\ell_2 \cdots \ell_n$.
The \textit{signed enhanced principal rank characteristic sequence} (\textit{sepr-sequence}) of $B$  is
the sequence $\sepr(B) = t_1t_2 \cdots t_n$, where
   \begin{equation*}
      t_k =
         \begin{cases}
             \tt{A^*} &\text{if $\ell_k = \tt{A}$
             and $B$ has both positive and negative order-$k$ principal minors;}\\
             \tt{A^+} &\text{if all the order-$k$ principal minors of $B$ are positive;}\\
             \tt{A^-} &\text{if all the order-$k$ principal minors of $B$ are negative;}\\
             \tt{N} &\text{if all the order-$k$ principal minors of $B$ are zero;}\\
             \tt{S^*} &\text{if $\ell_k = \tt{S}$
             and $B$ has both positive and negative order-$k$ principal minors;}\\
             \tt{S^+} &\text{if $\ell_k = \tt{S}$ and all the order-$k$ principal minors of $B$ are nonnegative;}\\
             \tt{S^-} &\text{if $\ell_k = \tt{S}$ and all the  order-$k$ principal minors of $B$ are nonpositive;}          
         \end{cases}
   \end{equation*}
   \end{defn}
\noindent
for all $j \in [n]$, $[\sepr(B)]_j:=t_j$.
%%%and 
%%%for all $p,q \in [n]$ with $p < q$, 
%%%$[\sepr(B)]_p^q := t_pt_{p+1} \cdots t_q$.
%%%For example,  
%%%$[\sepr(B)]_2^5 = t_2t_3t_4t_5$ (if $n \geq 5$). 

Let $\sigma = x_1x_2 \cdots x_n$ be a sequence from 
$\{\tt A,N,S\}$ or 
$\{\tt A^*,A^+,A^-,N,S^*,S^+,S^-\}$;
the \textit{order} of $\sigma$ is $n$;
we say that $\sigma$ is \textit{attainable} by a given matrix $B$ if 
$\epr(B) = \sigma$ or $\sepr(B) = \sigma$;
for all $p,q \in [n]$ with $p \leq q$,
$x_px_{p+1}x_{p+2} \cdots x_q$ is a \textit{subsequence} of $\sigma$;
if $\sigma$ does not occur as a subsequence of the 
epr- or sepr-sequence of any Hermitian (respectively, real symmetric) matrix, 
then $\sigma$ is said to be \textit{forbidden} (respectively, \textit{forbidden over $\R$}).

The objective of this paper is to answer the following two questions.

\begin{quest}\label{question-Hermitian}
\rm
Which sequences from 
$\{\tt A^*, A^+, A^-, N, S^*, S^+, S^-\}$ of orders 2 and 3 do not occur as a subsequence of the sepr-sequence of any Hermitian matrix? 
\end{quest}

\begin{quest}\label{question-real symmetric}
\rm
Which sequences from 
$\{\tt A^*, A^+, A^-, N, S^*, S^+, S^-\}$ of orders 2 and 3  do not occur as a subsequence of the sepr-sequence of any real symmetric matrix? 
\end{quest}

Observe that Question \ref{question-Hermitian}
(respectively, Question \ref{question-real symmetric})
may be stated as follows: 
Which sequences from 
$\{\tt A^*, A^+, A^-, N, S^*, S^+, S^-\}$ of orders 2 and 3  are forbidden (respectively, forbidden over $\R$)?

This paper expands  the list of sequences from 
$\{\tt A^*, A^+, A^-, N, S^*, S^+, S^-\}$  that are known to be forbidden, and it provides complete answers to Question \ref{question-Hermitian} and Question \ref{question-real symmetric}.

Before an attempt to answer the aforementioned two questions is made, we should verify whether or not the analog questions for epr-sequences of those two questions can be answered with what is currently in the literature.
It turns out that, as demonstrated below, we can.
%%%%
There are nine ($3^2=9$) sequences from $\{\tt A,N,S\}$ of order $2$.
It follows immediately from \cite[Table 3]{EPR} that each of those nine sequences is not forbidden. 
Furthermore, as every matrix listed or referenced in \cite[Table 3]{EPR} is real, each of those nine sequences is not forbidden over $\R$. 

There are twenty-seven ($3^3=27$) sequences from 
$\{\tt A,N,S\}$ of order $3$. 
Only three of those twenty-seven sequences are forbidden: $\tt NNA$, $\tt NNS$ and $\tt NSA$.

\begin{thm}\label{NNA NNS NSA}
{\rm \cite{EPR}}
A sequence from $\{\tt A,N,S\}$ of order $3$ does not occur as a subsequence of the epr-sequence of any Hermitian matrix 
if and only if 
it is one of the following sequences:
\[
\tt NNA, \  NNS \mbox{ \  and \ } NSA.
\]
\end{thm}

\begin{proof}
Let $\sigma$ be a sequence from $\{\tt A,N,S\}$ of order 3.

Suppose that 
$\sigma \in \{\tt NNA, \tt NNS, NSA\}$.
By \cite[Theorem 2.3]{EPR} and \cite[Corollary 2.7]{EPR}, $\sigma$ is forbidden.

Suppose that 
$\sigma \notin \{\tt NNA, \tt NNS, NSA\}$. 
If $\sigma \in \{\tt NAN, NAS\}$,
then $\sigma$ is a subsequence of one of the sequences that appears in 
\cite[Table 6.1]{EPR-Hermitian}, implying that
$\sigma$ is not forbidden.
If $\sigma \notin \{\tt NAN, NAS\}$,
then $\sigma$ is a subsequence of one of the sequences that appears in  
\cite[Table 4]{EPR}, implying that 
$\sigma$ is not forbidden. 
\end{proof}

Obviously, the three sequences in the statement of Theorem \ref{NNA NNS NSA} form a subset of the set of sequences from $\{\tt A,N,S \}$ of order $3$ that are forbidden over $\R$. By adding the sequences $\tt NAN$ and $\tt NAS$ to the former set we obtain the latter set.

\begin{thm}\label{NAN-NAS theorem}
{\rm \cite{EPR}}
A sequence  from $\{\tt A,N,S\}$ of order $3$ does not occur as a subsequence of the epr-sequence of any real symmetric matrix 
if and only if 
it is one of the following sequences:
\[
\tt NAN, \ NAS, \ NNA, \  NNS \mbox{ \  and \ } NSA.
\]
\end{thm}

\begin{proof}
Let $\sigma$ be a 
sequence from $\{\tt A,N,S\}$ of order 3.

Suppose that 
$\sigma \in \{\tt NAN, NAS, NNA, NNS, NSA\}$.
By \cite[Theorem 2.14]{EPR} and 
Theorem \ref{NNA NNS NSA},
$\sigma$ is forbidden over $\R$. 

Suppose that 
$\sigma \not\in 
\{\tt NAN, NAS, NNA, NNS, NSA\}$.
Observe that $\sigma$ is a subsequence of one of the sequences that appears in  
\cite[Table 4]{EPR}, implying that
$\sigma$ is not forbidden over $\R$.
\end{proof}

Observe that Theorems \ref{NNA NNS NSA} and \ref{NAN-NAS theorem}, respectively, provide answers to the analog questions for epr-sequences of Questions \ref{question-Hermitian} and \ref{question-real symmetric}.

We may now start our quest to answer Questions \ref{question-Hermitian} and \ref{question-real symmetric}. 
As seen below, what is currently in the literature is not enough to answer those two questions.
Theorems \ref{NNA NNS NSA} and \ref{NAN-NAS theorem} imply that their answers are not the same.
The natural first step in our quest is to make a list of the sequences from $\{\tt A^*, A^+, A^-, N, S^*, S^+, S^-\}$ of orders 2 and 3 that are known to be forbidden.

Unlike with sequences from $\{\tt A,N,S\}$ of order $2$, there are sequences from 
$\{\tt A^*, A^+, A^-, N, S^*, S^+, S^-\}$ of order $2$ that are forbidden.

\begin{thm}\label{sepr order-2 summary}
{\rm \cite[Theorem 3.4]{XMR-sepr}}
Neither $\tt A^*N$ nor $\tt NA^*$ occurs as a subsequence of the sepr-sequence of any Hermitian matrix.
\end{thm}

The next theorem provides a list of sequences of order 3 that are known to be forbidden.

\begin{thm}\label{sepr order-3 summary}
{\rm \cite{XMR-sepr}}
Let $\sigma$ be a sequence from 
$\{\tt A^*,A^+,A^-,N,S^*,S^+,S^-\}$ of order $3$.
If one of the following statements holds, then $\sigma$ does not occur as a subsequence of the sepr-sequence of any Hermitian matrix:

\begin{enumerate}
\item $\sigma \in \{\tt A^+XA^+,A^-XA^-,S^+XA^+,S^-XA^-\}$, 
for some $\tt X \in \{\tt A^*,N,S^*,S^+,S^-\}$.
\item $\sigma \in \{\tt A^+YS^+,A^-YS^-,S^+YS^+,S^-YS^-\}$,
for some $\tt Y \in \{\tt A^*,N\}$. 
\item $\sigma \in \{\tt S^*NN, S^*NZ, ZNS^*\}$,
for some $\tt Z \in \{\tt A^+,A^-,S^+,S^-\}$.
\end{enumerate}
\end{thm}

\begin{proof}
If statement (1) holds, then \cite[Theorems 3.5 and 3.10]{XMR-sepr} imply that $\sigma$ is forbidden.
%%%%%%
If statement (2) holds, then \cite[Propositions 3.8, 3.12 and 3.15]{XMR-sepr} imply that $\sigma$ is forbidden.
%%%%%%
If statement (3) holds, then \cite[Corollary 3.3]{XMR-sepr} and \cite[Propositions 3.9, 3.12 and 3.15]{XMR-sepr} imply that $\sigma$ is forbidden.
%%
%%If statement (4) or (5) hold, then Theorem \ref{sepr order-2 summary} or Theorem \ref{NNA NNS NSA} imply that $\sigma$ is forbidden over $\h$.
\end{proof}

Observe that, by Theorem \ref{sepr order-2 summary}, any sequence of order $3$ that contains either $\tt A^*N$ or $\tt NA^*$ as a subsequence is forbidden as well.
As implied above, there are sequences from 
$\{\tt A^*, A^+, A^-, N, S^*, S^+, S^-\}$ of orders 2 and 3 that are forbidden but are currently not known to be forbidden. 
Indeed, 
in our quest to answer Question \ref{question-Hermitian}, it is established that the following sequences are forbidden:
$\tt NS^*$, $\tt S^*N$,
$\tt A^+S^*S^+$, $\tt S^+S^*S^+$,      
$\tt A^-S^*S^-$, $\tt S^-S^*S^-$
(see Section \ref{s: Hermitian}).

This paper has five sections, and it unfolds as follows.
The part of this section following this paragraph introduces additional notation that we have adopted.
Section \ref{s: known facts} contains a list of known facts that are referenced in subsequent sections.
The purpose of Section \ref{s: lemmas} is to present preliminary results that are necessary to establish our main results.
Section \ref{s: Hermitian} is devoted to providing a (complete) answer to Question \ref{question-Hermitian}. 
In Section \ref{s: real symmetric}, we confine our attention to real symmetric matrices, and we provide a (complete) answer to Question \ref{question-real symmetric}.

Let $\sigma = x_1x_2 \cdots x_n$ be a sequence 
from 
$\{\tt A,N,S\}$ or 
$\{\tt A^*,A^+,A^-,N,S^*,S^+,S^-\}$.
If $\sigma$ is the epr-  or sepr-sequence of some matrix, then, for any integer $k$ with $1 \leq k \leq n$,
$x_1x_2 \cdots x_k$ is an \textit{initial} subsequence of $\sigma$. 
%%%and 
%%%$x_{k}x_{k+1} \cdots x_n$ is a \textit{terminal} %%%subsequence of $\sigma$.
The \textit{underlying sequence} of 
$x_1x_2 \cdots x_n$, 
which is denoted by $\uepr(x_1x_2 \cdots x_n)$,
is the sequence from 
$\{\tt A,N,S\}$ that results from removing the superscripts (if any) of the terms of $x_1x_2 \cdots x_n$. 
For example,
$\uepr({\tt A^+NS^-S^*}) = \tt ANSS$.
The {\it negative} of $x_1x_2 \cdots x_n$, 
which is denoted by $\negat(x_1x_2 \cdots x_n)$,
is the sequence that results from
replacing `+' superscripts with `-' superscripts
in $x_1x_2 \cdots x_n$, and vice versa.
For example, the negative of the sequence
${\tt S^-S^*A^*A^+N}$ is
${\tt S^+S^*A^*A^-N}$.

The conjugate of a complex number $a$ is denoted by $\overline{a}$.
For a given matrix $B$, the matrix that results from replacing each entry of $B$ with its conjugate is denoted by $\overline{B}$; 
and the conjugate transpose of $B$ is denoted by $B^*$ (i.e., $B^* = \overline{B}^T$).
With $O_n$ we denote the $n \times n$ zero matrix.
The $n\times n$ diagonal matrix whose 
$j$th diagonal entry is $d_j$, for all $j \in [n]$, is 
denoted by $\diag(d_1,d_2, \dots,d_n)$.
For two given matrices $B$ and $C$,
$B \oplus C$ denotes the direct sum of $B$ and $C$; that is, $B \oplus C$ is the block-diagonal matrix whose diagonal blocks are the matrices $B$ and $C$ (in that order). 
If the notation 
$R_{\sigma}$ (respectively, $C_{\sigma}$), 
for some sequence $\sigma$, is used to denote a matrix, then it is the case that 
$R_{\sigma}$ (respectively, $C_{\sigma}$) is real 
(respectively, non-real) and 
$\sepr(R_{\sigma})=\sigma$
(respectively, $\sepr(C_{\sigma})=\sigma$).

%%%%%%%%%%%%%%%%%%%%
%%%%%%%%%%%%%%%%%%%%
%%%%%%%%%%%%%%%%%%%%
%%%%%%%%%%%%%%%%%%%%
%%%%%%%%%%%%%%%%%%%%
%%%%%%%%%%%%%%%%%%%%
%%%%%%%%%%%%%%%%%%%%
%%%%%%%%%%%%%%%%%%%%
%%%%%%%%%%%%%%%%%%%%
\section{Known facts}\label{s: known facts}
$\null$
\indent
This section contains a list of known facts that are referenced in subsequent sections.

The following lemma is an elementary fact that follows immediately from the fact that appending a row and a column to a matrix of rank $r$ results in a matrix whose rank is at most $r+2$.

\begin{lem}
\label{rank when deleting}
Let $B \in \Cnn$ and $\rank(B) = r$.
Then the rank of any $(n-1)\times (n-1)$ submatrix of $B$ is at least $r-2$.
\end{lem}

Our next fact, which is well-known, asserts that the rank of a non-zero Hermitian matrix $B$ is equal to the order of a largest nonsingular principal submatrix of $B$.

\begin{thm}
%{\rm \cite[Theorem 1.1]{BIRS13}}
Let $B \in \Cnn$ be Hermitian.
Then $\rank(B) = \max \{ |\alpha| : \alpha \subseteq [n] \mbox{ and } \det (B[\alpha]) \neq 0 \}$
(where the maximum over the empty set is defined to be 0).
\end{thm}

For a given Hermitian matrix $B$, we shall reference the previous theorem by simply saying that `the rank of $B$ is principal.'

The simple observation that a simultaneous permutation of the rows and columns of a given matrix $B$ leaves $\sepr(B)$ invariant is one that is exploited here.

\begin{obs}\label{permutation}
Let $B \in \Cnn$ be Hermitian and
$P \in \Rnn$ be a permutation matrix. 
Then $PBP^T$ is Hermitian and 
$\sepr(PBP^T)=\sepr(B)$.
\end{obs}

The following fact is well-known 
(see, for example, \cite[p.\ 259]{Zhang}).

\begin{lem}\label{same sign lemma}
Let $B \in \Cnn$ be Hermitian and
$\rank(B) = r$.
Then all the nonzero principal minors of $B$ of
order $r$ have the same sign.
\end{lem}

%%%\begin{proof}
%%%As $B$ is Hermitian, $B$ has exactly
%%%$r$ nonzero eigenvalues, which are denoted by
%%%$\lambda_1, \lambda_2, \dots, \lambda_r$.
%%%We will now show that the sign of each nonzero principal minor of $B$ of order $r$ is equal to the sign of 
%%%$\prod_{j=1}^{r} \lambda_j$.
%%%Let $B'$ be a 
%%%nonsingular $r \times r$ principal submatrix of $B$
%%%(such a submatrix exists because the rank of $B$ is principal).
%%%As $B$ and $PBP^T$ have the same eigenvalues for any $n \times n$ permutation matrix $P$ (because they are similar), 
%%%we may assume, without loss of generality, that
%%%$B' = B[\{1,2, \dots, r\}]$.
%%%As $B$ is Hermitian, 
%%%there exists a unitary matrix $U$ such that
%%%$B = U^*DU$, where $D = \Lambda \oplus O_{n-r}$
%%%and $\Lambda = \diag(\lambda_1, \lambda_2, \dots, \lambda_r)$.
%%%Let $U_r = U[\{1,2, \dots, r\}]$.
%%%It is readily verified that $B'= U_r^* \Lambda U_r$.
%%%Thus,
%%%$\det(B') =
%%%|\det(U_r)|^2 \prod_{j=1}^{r} \lambda_j$,
%%%leading to the desired conclusion.
%%%Then, as $B'$ was arbitrary, 
%%%all the nonzero order-$r$ principal minors of $B$ have the same sign.
%%%\end{proof}

The following observation is exploited here as well.

\begin{obs}\label{odd terms obs}
{\rm \cite[Observation 2.2]{XMR-sepr}}
Let $B \in \Cnn$ be Hermitian, and
let $j$ be an integer with $1 \leq j \leq n$.
\begin{enumerate}
\item If $j$ is even,
then $[\sepr(-B)]_{j} = [\sepr(B)]_{j}$.
\item If $j$ is odd, then
$[\sepr(-B)]_{j}$ = $\negat( [\sepr(B)]_{j})$.
\end{enumerate}
\end{obs}

Observation \ref{odd terms obs} helps us minimize the number of matrix examples that have to be provided to accomplish our objective.
For example, if $\tt S^+S^-A^-A^-$ is attainable by  some Hermitian matrix $B$ (i.e., $\sepr(B) = \tt S^+S^-A^-A^-$), then, by Observation \ref{odd terms obs}, so is $\tt S^-S^-A^+A^-$, as $\sepr(-B) = \tt S^-S^-A^+A^-$.

The following theorem is the reason why the sequences $\tt NNA$ and $\tt NNS$ are forbidden  
(see Theorem \ref{NNA NNS NSA}).

\begin{thm}
{\rm ($\tt NN$ Theorem.)}
Let $B \in \Cnn$ be Hermitian, 
$\epr(B) = \ell_1 \ell_2 \cdots \ell_n$ and
$\sepr(B) = t_1t_2 \cdots t_n$.
\begin{enumerate}
\item[$(i)$] 
{\rm \cite[Theorem 2.3]{EPR}} 
If $\ell_{k} = \ell_{k+1} = \tt N$ for some $k$, 
then,  for all $j \geq k$, $\ell_j = \tt N$.

\item [$(ii)$] 
{\rm \cite[Theorem 2.3]{XMR-sepr}}
If $t_k = t_{k+1} = \tt{N}$ for some $k$,
then,  for all $j \geq k$, $t_j = \tt{N}$.
\end{enumerate}
\end{thm}

The following theorem is a consequence of Jacobi’s determinantal identity.

\begin{thm}
{\rm \cite[Theorem 2.4]{XMR-sepr} (Inverse Theorem.)}
Let $B \in \Cnn$ be nonsingular and Hermitian.
\begin{enumerate}
\item [$(i)$] 
If $\sepr(B) = t_1t_2 \cdots t_{n-1}\tt{A^+}$, then
$\sepr(B^{-1}) = t_{n-1}t_{n-2} \cdots t_1 \tt{A^+}$.

\item [$(ii)$] 
If $\sepr(B) = t_1t_2 \cdots t_{n-1}\tt{A^-}$, then
$\sepr(B^{-1})=
\negat(t_{n-1}t_{n-2} \cdots t_1)\tt{A^-}$.
\end{enumerate}
\end{thm}

The proofs of the items in the following theorem for which no reference is provided are trivial and, therefore, omitted here.

\begin{thm}\label{Inheritance}
{\rm (Inheritance Theorem.)}
Let $B \in \Cnn$ be Hermitian.
Let $m$ and $j$ be integers with $1\le j \le m \leq n$.
\begin{enumerate}
\item  {\rm \cite[Theorem 2.7 (1)]{XMR-sepr}} If $[\sepr(B)]_j={\tt N}$, then  $[\sepr(C)]_j={\tt N}$, for all $m\times m$ principal submatrices $C$.
\item {\rm \cite[Theorem 2.7 (2)]{XMR-sepr}} If  $[\sepr(B)]_j={\tt A^+}$, then  $[\sepr(C)]_j={\tt A^+}$, for all $m\times m$ principal submatrices $C$.
\item {\rm \cite[Theorem 2.7 (3)]{XMR-sepr}} If  $[\sepr(B)]_j={\tt A^-}$, then  $[\sepr(C)]_j={\tt A^-}$, for all $m\times m$ principal submatrices $C$.
%%%%%%%%%%%%%%%%%%  
%%%%%%%%%%%%%%%%%%  
%   The items below did not appear in 
%   the first sepr paper
%%%%%%%%%%%%%%%%%%  
%%%%%%%%%%%%%%%%%%  
\item  If $[\sepr(B)]_j= \tt S^+$, then  
$[\sepr(C)]_j \in \{\tt A^+,N, S^+\}$, for all $m\times m$ principal submatrices $C$.

\item  If $[\sepr(B)]_j= \tt S^-$, then  
$[\sepr(C)]_j \in \{\tt A^-,N, S^-\}$, for all $m\times m$ principal submatrices $C$.
\end{enumerate}
\end{thm}

Certain sequences do not occur as initial subsequences.

\begin{prop}\label{basic}{\rm \cite[Proposition 3.1]{XMR-sepr} (Basic Proposition.)}
None of the following sequences occurs as an initial subsequence of the sepr-sequence of any Hermitian matrix:
\vspace{-2mm}
\[
\tt A^*A^+, \
\tt A^*N,      \
\tt A^*S^+, \
\tt NA^*, \
\tt NA^+, \ 
\tt NS^*, \
\tt NS^+, \
\tt S^*A^+, \
\tt S^*N,      \
\tt S^*S^+, \
\tt S^+A^+ \mbox{  and \ }  
\tt S^-A^+.
\]
\end{prop}

This section concludes with a fact that is concerned with real symmetric matrices.

\begin{thm}\label{SNA}
{\rm \cite[Theorem 2.7]{XMR-Classif}}
Let $B \in \Rnn$ be symmetric and
$\epr(B) = \ell_1\ell_2 \cdots \ell_n$.
Then $\tt{SNA}$ is not a subsequence of
$\ell_1\ell_2 \cdots \ell_{n-2}$.
\end{thm}

%%%%%%%%%%%%%%%%%%%%
%%%%%%%%%%%%%%%%%%%%
%%%%%%%%%%%%%%%%%%%%
%%%%%%%%%%%%%%%%%%%%
%%%%%%%%%%%%%%%%%%%%
%%%%%%%%%%%%%%%%%%%%
%%%%%%%%%%%%%%%%%%%%
%%%%%%%%%%%%%%%%%%%%
%%%%%%%%%%%%%%%%%%%%
\section{Preliminary results}\label{s: lemmas}
$\null$
\indent
The purpose of this section is to present preliminary
results that are necessary to establish our main results.
In particular, those preliminary results are used to demonstrate that certain sequences are \textit{not} forbidden.

%%%%%%%%%%%%%%%%%%
%%%%%%%%%%%%%%%%%%
%% Real Matrices
%%%%%%%%%%%%%%%%%%
%%%%%%%%%%%%%%%%%%
The following observation is an analogous result for sepr-sequences of  \cite[Observation 2.19 (2)]{EPR}.

\begin{obs}\label{Append Zero}
Let $B \in \Cnn$ be  Hermitian and 
$\sepr(B)= t_1t_2 \cdots t_n$.
Then \ $\sepr(B \oplus O_1) = t'_1t'_2 \cdots t'_n{\tt N}$,
where, for $j=1,2, \dots, n$, all of the following statements hold:
%%%\begin{enumerate}
%%%\item If $t_j \in \{\tt N,S^*,S^+,S^-\}$,
%%%then $t'_j = t_j$.
%%%\item If $t_j = \tt A^*$, then $t'_j = \tt S^*$.
%%%\item If $t_j = \tt A^+$, then $t'_j = \tt S^+$.
%%%\item If $t_j = \tt A^-$, then $t'_j = \tt S^-$.
%%%\end{enumerate}
If $t_j \in \{\tt N,S^*,S^+,S^-\}$,
then $t'_j = t_j$;
if $t_j = \tt A^*$, then $t'_j = \tt S^*$;
if $t_j = \tt A^+$, then $t'_j = \tt S^+$;
if $t_j = \tt A^-$, then $t'_j = \tt S^-$.
\end{obs}

All the sequences that appear in the statement of the following fact are not only not forbidden, but also not forbidden over $\R$.

\begin{lem}\label{SSS etc}
Let $\sigma$ be any of the following sequences:
\begin{multicols}{4}
\begin{enumerate}
\item $\tt S^+NS^-$.
\item $\tt S^*S^+N$. \label{S... 113}
\item $\tt S^*S^*S^+$.
\item $\tt S^*S^*S^-$.
\item $\tt S^*S^-S^+$.
\item $\tt S^*S^-S^-$.
\item $\tt S^+S^*S^-$.
\item $\tt S^+S^+S^+$.
\item $\tt S^+S^+S^-$.
\item $\tt S^+S^-S^+$.
\item $\tt S^+S^-S^-$.
\item $\tt S^-S^*S^+$.
\item $\tt S^-S^+S^+$.
\item $\tt S^-S^+S^-$.
\item $\tt S^-S^-S^+$.
\item $\tt S^-S^-S^-$.
\end{enumerate}
\end{multicols}
\noindent
Then there exists a $4 \times 4$ real symmetric matrix $B$ such that $\sigma$ is a subsequence of $\sepr(B)$. 
\end{lem}
\begin{proof}
We divide the proof into two cases in terms of the sequence (\ref{S... 113}).

\noindent
\textbf{Case 1}: $\sigma = \tt S^*S^+N$.

\noindent
Observe that, if $B=\diag(1,-1,-1,0)$, then 
$\sepr(B) = \tt S^*S^*S^+N$.

\noindent
\textbf{Case 2}: $\sigma \neq \tt S^*S^+N$.

\noindent
Let $\psi$ be the sequence that results from replacing each $\tt S$ in $\sigma$ with $\tt A$. 
Observe that $\psi$ appears in \cite[Table 2]{XMR-sepr}, implying that there exists a ($3 \times 3$) Hermitian matrix $H$ such that $\sepr(H) = \psi$. 
As the only non-real Hermitian matrix listed in \cite[Table 2]{XMR-sepr} is $M_{\tt NA^-N}$,
and because $\psi \neq \tt NA^-N$,
we may assume that $H$ is a real symmetric matrix.
Let $H' = H \oplus O_1$.
As $\sigma$ is a sequence from 
$\{\tt N, S^*, S^+, S^-\}$,
and because of the relationship between $\psi$ and $\sigma$, Observation \ref{Append Zero} implies that
$\sepr(H') = \sigma {\tt N}$.
The desired conclusion, then, follows by letting $B=H'$.
\end{proof}

The following observation is an analogous result for sepr-sequences of  \cite[Observation 2.19 (1)]{EPR}.

\begin{obs}\label{Append LRC}
Let $B=[b_{ij}] \in \Cnn$ be Hermitian and 
$\sepr(B)= t_1t_2 \cdots t_n$.
Let $y$ be the last column vector of $B$ and
\[
B'=
\left(
\begin{array}{c|c}
 B & y \\
\hline \\[-4mm]
\overline{y}^T & b_{nn}
\end{array}
\right) .               
\]
Then $\sepr(B') = t_1t'_2 \cdots t'_n{\tt N}$,
where, for $j=2, 3, \dots, n$, all of the following statements hold:
%%%\begin{enumerate}
%%%\item If $t_j \in \{\tt N,S^*,S^+,S^-\}$,
%%%then $t'_j = t_j$.
%%%\item If $t_j = \tt A^*$, then $t'_j = \tt S^*$.
%%%\item If $t_j = \tt A^+$, then $t'_j = \tt S^+$.
%%%\item If $t_j = \tt A^-$, then $t'_j = \tt S^-$.
%%%\end{enumerate}
\item If $t_j \in \{\tt N,S^*,S^+,S^-\}$,
then $t'_j = t_j$;
if $t_j = \tt A^*$, then $t'_j = \tt S^*$;
if $t_j = \tt A^+$, then $t'_j = \tt S^+$;
if $t_j = \tt A^-$, then $t'_j = \tt S^-$.
\end{obs}

As with the sequences in Lemma \ref{SSS etc}, all the sequences that appear in the statement of the following fact are not forbidden over $\R$.

\begin{lem}\label{ASS etc}
Let $\sigma$ be any of the following sequences:

\begin{multicols}{4}
\begin{enumerate}
\item $\tt A^*S^*S^+$.
\item $\tt A^*S^*S^-$.
\item $\tt A^*S^-S^+$.
\item $\tt A^*S^-S^-$.
\item $\tt A^+S^*S^-$.
\item $\tt A^+S^+S^+$.
\item $\tt A^+S^+S^-$.
\item $\tt A^+S^-S^+$.
\item $\tt A^+S^-S^-$.
\item $\tt A^-S^*S^+$.
\item $\tt A^-S^+S^+$.
\item $\tt A^-S^+S^-$.
\item $\tt A^-S^-S^+$.
\item $\tt A^-S^-S^-$.
\item $\tt NS^-S^+$.
\item $\tt NS^-S^-$.
\end{enumerate}
\end{multicols}
\noindent
Then there exists a $4 \times 4$ real symmetric matrix $B$ such that $\sigma$ is a subsequence of $\sepr(B)$. 
\end{lem}
\begin{proof}
Let $\psi$ be the sequence that results from replacing each $\tt S$ in $\sigma$ with $\tt A$. 
As there is no $\tt S$ in the first term of $\sigma$, the first term of $\sigma$ and $\psi$ are equal to each other.
Observe that $\psi$ appears in \cite[Table 2]{XMR-sepr}, implying that there exists a ($3 \times 3$) Hermitian matrix $H$ such that $\sepr(H) = \psi$. 
As the only non-real Hermitian matrix listed in \cite[Table 2]{XMR-sepr} is $M_{\tt NA^-N}$,
and because $\psi \neq \tt NA^-N$,
we may assume that $H$ is a real symmetric matrix.
Let $H=[h_{ij}]$, $y$ be the last column vector of $H$ and
\[
H'=
\left(
\begin{array}{c|c}
 H & y \\
\hline \\[-4mm]
y^T & h_{33}
\end{array}
\right) .               
\]
As the first term of $\sigma$ and $\psi$ are equal to each other,
and because the second and third terms of $\sigma$ are in $\{\tt S^*, S^+, S^-\}$,
the relationship between $\psi$ and $\sigma$ implies that we can use Observation \ref{Append LRC} to reach the following conclusion: 
$\sepr(H') = \sigma {\tt N}$.
The desired conclusion, then, follows by letting $B=H'$.
\end{proof}

The remainder of this section consists of a list of lemmas that are readily verified and whose proof has been omitted. 
 
\begin{lem}\label{VierOne}
Let
\begin{equation*}
F(x,y,a)=
\left(
\begin{array}{rr|rr}
 x & 5 & 1 & 1 \\
 5 & y & 1 & 1 \\
 \hline
 1 & 1 & 1 & a \\
 1 & 1 & a & 1 \\
\end{array}
\right).
\end{equation*}
For the following matrices, 
$\sepr(R_{\sigma}) = \sigma$:
\begin{multicols}{2}
 \begin{enumerate}
 \item $\VierOneA :=F(2,\frac{1}{2},2)$.\label{VierOneA}
 \item $\VierOneB:=F(6,6,-\frac{3}{4})$.\label{VierOneB}
 \item $\VierOneC:=F(\frac{1}{2},\frac{1}{2},2)$.  \label{VierOneC}
 \item $\VierOneD:=F(-6,-5,-\frac{47}{5})$.\label{VierOneD}
 \item $\VierOneE:=F(4,5,-\frac{3}{5})$. \label{VierOneE}
 \item $\VierOneF:=F(-5,-5,0)$. \label{VierOneF}
 \item $\VierOneG:=F(1,3,0)$. \label{VierOneG}
 \item $\VierOneH:=F(1,25,\frac{1}{2})$. \label{VierOneH}
 \item $\VierOneI:=F(\frac{5}{2},10,-\frac{9}{10})$. \label{VierOneI}
 \item $\VierOneJ:=F(1,\frac{3}{4},\frac{4}{3})$. \label{VierOneJ}   
 \end{enumerate}
\end{multicols}
\end{lem}

\begin{lem}\label{VierTwo}
Let
\begin{equation*}
F(x,y,a)=
\left(
\begin{array}{rr|rr}
 -2 & 2 & 2 & 0 \\
 2 & -2 & 0 & 2 \\
 \hline
 2 & 0 & x & a \\
 0 & 2 & a & y \\
\end{array}
\right).
\end{equation*}
For the following matrices, 
$\sepr(R_{\sigma}) = \sigma$:
\begin{multicols}{2}
 \begin{enumerate}
 \item $\VierTwoA := F(-2,2,1)$. \label{VierTwoA} 
 \item $\VierTwoB := F(2,2,3)$. \label{VierTwoB} 
 \item $\VierTwoC := F(-2,-2,-1)$. \label{VierTwoC} 
 \end{enumerate}
\end{multicols}
\end{lem}

\begin{lem}\label{VierThree}
Let
\begin{equation*}
F(x,y,z,a)=
\left(
\begin{array}{rr|rr}
 x & 1 & 1 & -1 \\
 1 & y & 1 & 1 \\
 \hline
 1 & 1 & z & a \\
 -1 & 1 & a & 0 \\
\end{array}
\right).
\end{equation*}
For the following matrices, 
$\sepr(R_{\sigma}) = \sigma$:
\begin{multicols}{2}
 \begin{enumerate}
 \item $\VierThreeA := F(0,0,0,1)$. \label{VierThreeA} 
 \item $\VierThreeB := F(0,0,0,0)$. \label{VierThreeB} 
 \item $\VierThreeC := F(2,2,4,1)$. \label{VierThreeC} 
 \end{enumerate}
\end{multicols}
\end{lem}

\begin{lem}\label{VierFour}
Let
\begin{equation*}
F(x,a,b)=
\left(
\begin{array}{rr|rr}
 0 & a & 2 & 1 \\
 a & x & 1 & 2 \\
 \hline
 2 & 1 & 1 & b \\
 1 & 2 & b & 1 \\
\end{array}
\right).
\end{equation*}
For the following matrices, 
$\sepr(R_{\sigma}) = \sigma$:
\begin{multicols}{2}
 \begin{enumerate}
 \item $\VierFourA :=F(-1,1,0)$.\label{VierFourA}
 \item $\VierFourB:=F(-1,2,5)$.\label{VierFourB}
 \item $\VierFourC:=F(-1,2,2)$.\label{VierFourC}
 \item $\VierFourD:=F(0,1,0)$.\label{VierFourD}
 \item $\VierFourE:=F(0,\frac{1}{2},2)$.\label{VierFourE}
 \item $\VierFourF:=F(0,2,2)$. \label{VierFourF}
 \item $\VierFourG:=F(0,5,-2)$. \label{VierFourG}
 \item $\VierFourH:=F(0,-1,-2)$. \label{VierFourH}
 \item $\VierFourI:=F(-1,2+\sqrt{5},0)$. \label{VierFourI}
 \item $\VierFourJ:=F(0,4,0)$. \label{VierFourJ} 
 \item $\VierFourK:=F(-1,0,0)$. \label{VierFourK}  
 \item $\VierFourL:=F(-1,0,2)$. \label{VierFourL}  
 \item $\VierFourM:=F(-1,0,4)$. \label{VierFourM}  
 \item $\VierFourN:=F(1,3,0)$. \label{VierFourN}  
 \item $\VierFourO:=F(1,3,1)$. \label{VierFourO}    
 \end{enumerate}
\end{multicols}
\end{lem}

%$\VierFourO$
%
%Lemma \ref{VierFour} (\ref{VierFourO})

\begin{lem}\label{VierFive}
Let
\begin{equation*}
F(x,a,b)=
\left(
\begin{array}{rr|rr}
 0 & a & 1 & 2 \\
 a & -1 & 1 & 1 \\
 \hline
 1 & 1 & 1 & b \\
 2 & 1 & b & x \\
\end{array}
\right).
\end{equation*}
For the following matrices, 
$\sepr(R_{\sigma}) = \sigma$:
\begin{multicols}{2}
 \begin{enumerate}
 \item $\VierFiveA := F(-2,1,10)$. \label{VierFiveA} 
 \item $\VierFiveB := F(-5,1,1)$. \label{VierFiveB} 
 \item $\VierFiveC := F(-2,2,\frac{1}{2})$. \label{VierFiveC} 
 \item $\VierFiveD := F(-2,0,\frac{3}{5})$. \label{VierFiveD} 
 \end{enumerate}
\end{multicols}
\end{lem}

\begin{lem}\label{VierSix}
Let
\begin{equation*}
F(x,a)=
\left(
\begin{array}{rr|rr}
 0 & 0 & 1 & 0 \\
 0 & -1 & 0 & 1 \\
 \hline
 1 & 0 & 1 & a \\
 0 & 1 & a & x \\
\end{array}
\right).
\end{equation*}
For the following matrices, 
$\sepr(R_{\sigma}) = \sigma$:
\begin{multicols}{2}
 \begin{enumerate}
 \item $\VierSixA := F(4,1)$. \label{VierSixA} 
 \item $\VierSixB := F(-5,2)$. \label{VierSixB}  
 \end{enumerate}
\end{multicols}
\end{lem}

\begin{lem}\label{FiveOne}
Let
\begin{equation*}
F(a,b)=
\left(
\begin{array}{rrr|rr}
 1 & a & 2 & 1 & 1 \\
 a & 1 & 2 & 1 & 1 \\
 2 & 2 & 1 & 1 & 1 \\
 \hline
 1 & 1 & 1 & -1 & b \\
 1 & 1 & 1 & b & -1 \\
\end{array}
\right)^{-1}.
\end{equation*}
For the following matrices, 
$\sepr(R_{\sigma}) = \sigma$:
\begin{multicols}{2}
 \begin{enumerate}
 \item $\FiveOneA := F(-2,2)$.\label{FiveOneA} 
 \item $\FiveOneB := F(-2,1)$.\label{FiveOneB} 
 \item $\FiveOneC := F(-3,2)$.\label{FiveOneC} 
 \item $\FiveOneD := F(7,2)$.\label{FiveOneD} 
 \item $\FiveOneE := F(-3,1)$.\label{FiveOneE} 
 \item $\FiveOneF := F(7,1)$.\label{FiveOneF} 
 \end{enumerate}
\end{multicols}
\end{lem}

\begin{lem}\label{FiveTwo}
Let 
\begin{equation*}
F(x,y,a,b,c) = 
\left(
\begin{array}{rrr|rr}
 x & a & b & 1 & 1 \\
 a & y & c & 1 & 1 \\
 b & c & 0 & 1 & 1 \\
 \hline
 1 & 1 & 1 & 0 & 1 \\
 1 & 1 & 1 & 1 & 0 \\
\end{array}
\right)^{-1}.
\end{equation*}
For the following matrices, 
$\sepr(R_{\sigma}) = \sigma$:
\begin{multicols}{2}
 \begin{enumerate}
 \item $\FiveTwoA := F(1,-1,-1,1,1)$.\label{FiveTwoA} 
 \item $\FiveTwoB := F(-1,0,-1,1,1)$.\label{FiveTwoB} 
 \item $\FiveTwoC := F(1,0,1,1,3)$.\label{FiveTwoC} 
 \item $\FiveTwoD := F(1,-1,-1,1,0)$.\label{FiveTwoD} 
 \item $\FiveTwoE := F(-1,0,0,1,-1)$.\label{FiveTwoE} 
 \item $\FiveTwoF := F(1,0,0,1,1)$.\label{FiveTwoF} 
 \item $\FiveTwoG := F(1,-1,-1,\frac{1}{2},1)$.\label{FiveTwoG} 
 \item $\FiveTwoH := F(-1,0,-1,1,2)$.\label{FiveTwoH} 
 \item $\FiveTwoI := F(1,0,\frac{1}{2},1,-1)$.\label{FiveTwoI} 
 \item $\FiveTwoJ := F(1,-1,1,\frac{1}{2},1)$.\label{FiveTwoJ} 
 \item $\FiveTwoK := F(-1,0,-\frac{1}{2},1,1)$.\label{FiveTwoK} 
 \item $\FiveTwoL := F(1,0,\frac{1}{2},1,1)$.\label{FiveTwoL} 
 \item $\FiveTwoM := F(1,-1,-1,0,0)$.\label{FiveTwoM} 
 \item $\FiveTwoN := F(-1,0,0,-1,0)$.\label{FiveTwoN} 
 \item $\FiveTwoO := F(1,0,0,1,0)$.\label{FiveTwoO} 
 \item $\FiveTwoP := F(1,-1,0,1,0)$.\label{FiveTwoP} 
 \end{enumerate}
 \end{multicols}
\end{lem}

\begin{lem}\label{SixOne}
Let
\begin{equation*}
F(x,y)=
\left(
\begin{array}{rrr|rrr}
 1 & -9 & -2 & 5 & 2 & 2 \\
 -9 & x & -2 & 2 & 5 & 2 \\
 -2 & -2 & y & 2 & 2 & 5 \\
 \hline
 5 & 2 & 2 & 1 & -9 & -2 \\
 2 & 5 & 2 & -9 & 1 & -2 \\
 2 & 2 & 5 & -2 & -2 & -1 \\
\end{array}
\right)^{-1}.
\end{equation*}
For the following matrices, 
$\sepr(R_{\sigma}) = \sigma$:
\begin{multicols}{2}
 \begin{enumerate}
 \item $\SixOneA := F(1,-1)$.\label{SixOneA} 
 \item $\SixOneB := F(1,0)$.\label{SixOneB} 
 \item $\SixOneC := F(4,-1)$.\label{SixOneC} 
 \item $\SixOneD := F(4,0)$.\label{SixOneD} 
 \end{enumerate}
\end{multicols}
\end{lem}

\begin{lem}\label{NSFreal}
For the following matrices, 
$\sepr(R_{\sigma}) = \sigma$:

\begin{enumerate}
%%%%%%%%%%%%%%%%%%
%%%%%%%%%%%%%%%%%%
\item
$\NSFrealA :=
\left(
\begin{array}{rrrr}
 1 & 0 & 1 & 1 \\
 0 & 1 & -1 & 1 \\
 1 & -1 & 1 & 0 \\
 1 & 1 & 0 & 1 \\
\end{array}
\right)$. \label{NSFrealA} 
%%%%%%%%
%%%%%%%%
\item
$\NSFrealB :=
\left(
\begin{array}{rrrrr}
 -1 & 1 & 1 & 1 & 1 \\
 1 & -1 & -1 & 1 & 1 \\
 1 & -1 & -1 & -1 & 1 \\
 1 & 1 & -1 & -1 & -1 \\
 1 & 1 & 1 & -1 & -1 \\
\end{array}
\right)$. \label{NSFrealB} 
%%%%%%%%
%%%%%%%%
\item 
$\NSFrealC :=
\left(
\begin{array}{rrrrr}
 -5 & 5 & -5 & 1 & -1 \\
 5 & -9 & -1 & 1 & 3 \\
 -5 & -1 & -9 & 3 & 1 \\
 1 & 1 & 3 & -1 & -1 \\
 -1 & 3 & 1 & -1 & -1 \\
\end{array}
\right)$. \label{NSFrealC} 
%%%%%%%%%%%%%%%%%%
%%%%%%%%%%%%%%%%%%
\item 
$\NSFrealD :=
\left(
\begin{array}{rrrrrr}
 -2 & -7 & -9 & 9 & 18 & 18 \\
 -7 & -21 & -28 & 28 & 56 & 56\\
 -9 & -28 & -37 & 38 & 76 & 76\\
 9 & 28 & 38 & -40 & -79 & -81\\
 18 & 56 & 76 & -79 & -156 & -159 \\
 18 & 56 & 76 & -81 & -159 & -162 \\
\end{array}
\right)$. \label{NSFrealD} 
%%%%%%%%
%%%%%%%%
\item
$\NSFrealE :=
\left(
\begin{array}{rrrrrr}
 0 & 0 & 1 & 0 & 0 & 1 \\
 0 & 0 & 0 & 1 & 0 & 1 \\
 1 & 0 & 0 & 1 & 0 & 0 \\
 0 & 1 & 1 & 0 & 1 & 0 \\
 0 & 0 & 0 & 1 & 0 & -1 \\
 1 & 1 & 0 & 0 & -1 & 0 \\
\end{array}
\right)$. \label{NSFrealE} 
%%%%%%%%
%%%%%%%%
\item 
$\NSFrealF :=
\left(
\begin{array}{rrrrrr}
 0 & 0 & 1 & 0 & 0 & 1 \\
 0 & 0 & 0 & 1 & 0 & 1 \\
 1 & 0 & 0 & 1 & 0 & 0 \\
 0 & 1 & 1 & 0 & 1 & 0 \\
 0 & 0 & 0 & 1 & 0 & 2 \\
 1 & 1 & 0 & 0 & 2 & 0 \\
\end{array}
\right)$. \label{NSFrealF} 
\end{enumerate}
\end{lem}

%%%%%%%%%%%%%%%%%%
%%%%%%%%%%%%%%%%%%
%% Complex Matrices
%%%%%%%%%%%%%%%%%%
%%%%%%%%%%%%%%%%%%

\begin{lem}\label{Complex}
Let
\begin{equation*}
F(a,b)=
\left(
\begin{array}{rrr|rrr}
 0               &    0         & i & -i & 1 & a \\
 0               &    0         & 0 & b & 1 & -i \\
 -i               &    0         & 0 & 0 & 1 & 1 \\
 \hline
 i                &  \OL{b}   & 0 & 0 & 0 & -2 \\
 1               &    1         & 1 & 0 & 0 & 0 \\
 \OL{a}       &    i          & 1 & -2 & 0 & 0 \\
\end{array}
\right)
\quad \text{ and } \quad \
G(a)=
\left(
\begin{array}{rr|rr}
 0 & 1 &    1       & 1 \\
 1 & 0 &    i        & 1 \\
 \hline
 1 & -i &    0       & a \\
 1 & 1 & \OL{a}  & 0 \\
\end{array}
\right).
\end{equation*}
For the following matrices, 
$\sepr(C_{\sigma}) = \sigma$:
\begin{multicols}{2}
 \begin{enumerate}
 \item $\ComOneA := F(-4,1)$.\label{ComOneA} 
 \item $\ComOneB := F(2,-1)$.\label{ComOneB}
 \item $\ComTwoA := G(i)$.\label{ComTwoA} 
 \item $\ComTwoB := G(-i)$.\label{ComTwoB}
 \end{enumerate}
\end{multicols}
\end{lem}

\begin{lem}\label{NSFcom}
For the following matrices, 
$\sepr(C_{\sigma}) = \sigma$:
\begin{enumerate}
\item 
$\NSFcomA :=
\left(
\begin{array}{rrrrr}
 -1 & 1 & 1 & 1 & 1 \\
 1 & -1 & 1 & 1 & 1 \\
 1 & 1 & -1 & -i & i \\
 1 & 1 & i & -1 & -i \\
 1 & 1 & -i & i & -1 \\
\end{array}
\right)$. \label{NSFcomA} 
%%%%%%%%
%%%%%%%%
\item 
$\NSFcomB :=
\left(
\begin{array}{rrrrr}
 0 & i & i & i & i \\
 -i & 0 & i & i & i \\
 -i & -i & 0 & i & i \\
 -i & -i & -i & 0 & i \\
 -i & -i & -i & -i & 0 \\
\end{array}
\right)$. \label{NSFcomB} 
\end{enumerate}
\end{lem}

%%%%%%%%%%%%%%%%%%%%
%%%%%%%%%%%%%%%%%%%%
%%%%%%%%%%%%%%%%%%%%
%%%%%%%%%%%%%%%%%%%%
%%%%%%%%%%%%%%%%%%%%
%%%%%%%%%%%%%%%%%%%%
%%%%%%%%%%%%%%%%%%%%
%%%%%%%%%%%%%%%%%%%%
%%%%%%%%%%%%%%%%%%%%
\section{Hermitian matrices}\label{s: Hermitian}
$\null$
\indent
In this section, we expand the list of sequences from $\{\tt A^*, A^+, A^-, N, S^*, S^+, S^-\}$ of order 2, as well as the list of sequences of order 3, that are known to be forbidden, and we provide a (complete) answer to Question \ref{question-Hermitian}
(see Theorem \ref{main result for order 2} and Theorem \ref{main result for order 3}).
%%%In particular, from Theorem \ref{main result for order 3}, we obtain all the sequences from
%%%$\{\tt A^*, A^+, A^-, N, S^*, S^+, S^-\}$ of order $3$ that are not forbidden;
%%%those sequences are listed in Table \ref{long table}. 

The first result of this section states that the sequences $\tt NS^*$ and $\tt S^*N$ are forbidden.

\begin{thm}\label{S*N and NS*}
Neither $\tt NS^*$ nor $\tt S^*N$ occurs as a subsequence of the sepr-sequence of any Hermitian matrix.
\end{thm}

\begin{proof}
We start by demonstrating that $\tt S^*N$ is forbidden.
Suppose on the contrary that there exists a Hermitian matrix $B$ such that
$\sepr(B) := t_1t_2 \cdots t_n$ and
$t_kt_{k+1} = \tt S^*N$, 
for some integer $k$ with $1 \leq k \leq n-1$.
Let $m$ be the minimal integer with  
$k+1 \leq m \leq n$ 
for which 
$B$ has an $m \times m$ principal submatrix $C$ such that $\sepr(C) := c_1c_2 \cdots c_m$ and $c_{k} =  \tt S^*$.  
By the Inheritance Theorem, $c_{k+1} =  \tt N$.
As $c_k \neq \tt N$, $\rank(C) \geq k$.
Then, as $c_{k} =  \tt S^*$,
Lemma \ref{same sign lemma} implies that
$\rank(C) \geq k+1$.
Then, as $c_{k+1} =  \tt N$,
and because the rank of $C$ is principal, 
$\rank(C) \geq k+2$.
Then, as $c_{k+1} =  \tt N$, 
and because the rank of $C$ is principal,
the $\tt NN$ Theorem implies that 
$c_{k+2} \neq \tt N$.
By Theorem \ref{sepr order-2 summary},
$c_{k+2} \neq \tt A^*$.
Then, as 
$c_kc_{k+1}c_{k+2} = {\tt S^*N}c_{k+2}$, 
Theorem \ref{sepr order-3 summary} implies that $c_{k+2} = \tt S^*$.
Thus, $c_kc_{k+1}c_{k+2} = \tt S^*NS^*$.
As no matrix has an sepr-sequence whose 
last term is $\tt S^*$, 
the order of $C$ is at least $k+3$; that is, 
$m \geq k+3$.  Thus, $m \geq 4$.

As $c_k = \tt S^*$, there exist sets 
$\alpha, \beta \subseteq \{1,2, \dots, m\}$ such that 
$|\alpha|=|\beta|=k$,
$\det(C[\alpha]) > 0$ and 
$\det(C[\beta]) < 0$.
Because of Observation \ref{permutation},
we may assume, without loss of generality, that
$\alpha = \{1,2, \dots, k\}$ and 
$\beta \subseteq \{2, 3, \dots, m\}$.
Let $Z=C[\{1,2, \dots, m\} \setminus \{1,m\}]$ and
$\sepr(Z)=z_1z_2 \cdots z_{m-2}$.
Observe that, since $m \geq k+3$, the order of $Z$ (i.e., $m-2$) is at least $k+1$.
%%%that is, $m-2 \geq k+1$.
By the Inheritance Theorem, 
$z_{k+1} = c_{k+1} = \tt N$.

\noindent
{\bf Claim 1}: $z_{k} = \tt N$.

\noindent
We will now establish Claim 1.
Let $M$ be a $k \times k$ principal submatrix of $Z$.
We will now prove that $\det(M)=0$.
Let 
$X = C[\{1,2, \dots, m-1\}]$,
$Y = C[\{2,3, \dots, m\}]$,
$\sepr(X) = x_1x_2 \cdots x_{m-1}$ and
$\sepr(Y) = y_1y_2 \cdots y_{m-1}$.
By the Inheritance Theorem, 
$x_{k+1} = c_{k+1} = \tt N$ and 
$y_{k+1} = c_{k+1} = \tt N$.
Observe that, since $m \geq k+3$, 
$X$ and $Y$ are of order at least $k+2$.
%%%that is, $m-1 \geq k+2$.
Then, as $X$ and $Y$ are principal submatrices of 
(not only $C$ but also) $B$,
and because their order (i.e., $m-1$) 
is at least $k+1$ but less than $m$,
the minimality of $m$ implies that 
$x_k \neq \tt S^*$ and $y_k \neq \tt S^*$.
By Theorem \ref{sepr order-2 summary},
$x_k \neq \tt A^*$ and $y_k \neq \tt A^*$.
Thus, $x_k, y_k \in \{\tt A^+, A^-, N, S^+, S^-\}$.
Note that 
$C[\alpha]$ is a principal submatrix of $X$, and that
$C[\beta]$ is a principal submatrix of $Y$.
Then, as $\det(C[\alpha])>0$ and $\det(C[\beta]) <0$,
$x_k \in \{\tt A^+, S^+\}$ and $y_k \in \{\tt A^-, S^-\}$.
Observe that
$M$ is a $k \times k$ principal submatrix of $X$.
Then, as $x_k \in \{\tt A^+, S^+\}$, 
$\det(M) \geq 0$.
Observe, as well, that 
$M$ is a $k \times k$ principal submatrix of $Y$.
Then, as $y_k \in \{\tt A^-, S^-\}$,
$\det(M) \leq 0$. 
Thus, $\det(M) = 0$, as desired.
It follows, then, that Claim 1 is true.

\noindent
{\bf Claim 2}: $\rank(C) \leq k+2$.

\noindent
We will now establish Claim 2.
Recall that 
$Z=C[\{1,2, \dots, m\} \setminus \{1,m\}]$, 
$X = C[\{1,2, \dots, m-1\}]$ and 
$C$ is $m \times m$.
%%%$C=C[\{1,2, \dots, m\}]$.
By Lemma \ref{rank when deleting},
$\rank(X) \leq \rank(Z) + 2$ and 
$\rank(C) \leq \rank(X)+2$.
As $z_{k+1} = \tt N$,
and because $z_{k} = \tt N$ (by Claim 1),
the $\tt NN$ Theorem implies that 
$z_j = \tt N$ for all $j \geq k$.
Then, as the rank of $Z$ is principal, 
$\rank(Z) \leq k-1$.
Then, as $\rank(X) \leq \rank(Z) + 2$,
$\rank(X) \leq k+1$. 
Then, as $x_{k+1} = \tt N$, 
and because the rank of $X$ is principal,
$\rank(X) \leq k$.
Then, as $\rank(C) \leq \rank(X)+2$,
$\rank(C) \leq k +2$.
Thus, Claim 2 is true.

As Claim 2 is true, and because 
$\rank(C) \geq k+2$, $\rank(C) = k+2$.
Then, as $c_{k+2} = \tt S^*$, 
we have reached a contradiction to 
Lemma \ref{same sign lemma}.
Hence, $\tt S^*N$ is forbidden.

To demonstrate that $\tt NS^*$ is forbidden, suppose on the contrary that there exists a Hermitian matrix $H$ such that 
$\sepr(H) := h_1h_2 \cdots h_n$ and 
$h_kh_{k+1} = \tt NS^*$, 
for some integer $k$ with $1 \leq k \leq n-1$.
By the Basic Proposition, $k \geq 2$.
By the $\tt NN$ Theorem, $h_{k-1} \neq \tt N$.
By Theorem \ref{sepr order-2 summary}, 
$h_{k-1} \neq \tt A^*$.
Then, as 
$h_{k-1}h_{k}h_{k+1} = h_{k-1}{\tt NS^*}$,
Theorem \ref{sepr order-3 summary} implies that $h_{k-1} = \tt S^*$.
Thus, $h_{k-1}h_k = \tt S^*N$,
contradicting what was established before this paragraph.
\end{proof}

The following theorem provides one part of the answer to Question \ref{question-Hermitian}.
Its proof references Table \ref{long table};
it is left to the reader to verify that in doing so
no circular reasoning took place.

\begin{thm}\label{main result for order 2}
A sequence from 
$\{\tt A^*,A^+,A^-,N,S^*,S^+,S^-\}$ of order $2$ does not occur as a subsequence of the sepr-sequence of any Hermitian matrix
if and only if 
it is one of the following sequences:
\[
\tt A^*N, \ NA^*, \ NS^* \mbox{ and \ } S^*N.
\]
%Moreover, for each sequence of order 2 that is not one of those four sequences, there exists a real symmetric matrix whose sepr-sequence contains it as a subsequence.
\end{thm}

\begin{proof}
Observe that there are $7^2=49$ sequences from  
$\{\tt A^*,A^+,A^-,N,S^*,S^+,S^-\}$ of order $2$.
Theorem \ref{sepr order-2 summary} and 
Theorem \ref{S*N and NS*} imply that
$\tt A^*N$, $\tt NA^*$, $\tt S^*N$ and $\tt NS^*$ are forbidden. 
We claim that the remaining 45 sequences are not forbidden. 
Observe that $\sepr(O_2)= \tt NN$;
thus, $\tt NN$ is not forbidden.
The remaining 44 sequences are the following:
$\tt A^* A^*$, 
$\tt A^* A^+$,
$\tt A^* A^-$,
$\tt A^+ A^*$,
$\tt A^+ A^+$,
$\tt A^+ A^-$,
$\tt A^- A^*$,
$\tt A^- A^+$,
$\tt A^- A^-$,
%$\tt A^* N$,
$\tt A^+ N$,
$\tt A^- N$,
$\tt A^* S^*$,
$\tt A^* S^+$,
$\tt A^* S^-$,
$\tt A^+ S^*$,
$\tt A^+ S^+$,
$\tt A^+ S^-$,
$\tt A^- S^*$,
$\tt A^- S^+$,
$\tt A^- S^-$,
%$\tt N A^*$,
$\tt N A^+$,
$\tt N A^-$,
%$\tt NN$,
%$\tt N S^*$,
$\tt N S^+$,
$\tt N S^-$,
$\tt S^* A^*$,
$\tt S^* A^+$,
$\tt S^* A^-$,
$\tt S^+ A^*$,
$\tt S^+ A^+$,
$\tt S^+ A^-$,
$\tt S^- A^*$,
$\tt S^- A^+$,
$\tt S^- A^-$,
%$\tt S^* N$,
$\tt S^+ N$,
$\tt S^- N$,
$\tt S^* S^*$,
$\tt S^* S^+$,
$\tt S^* S^-$,
$\tt S^+ S^*$,
$\tt S^+ S^+$,
$\tt S^+ S^-$,
$\tt S^- S^*$,
$\tt S^- S^+$ and
$\tt S^- S^-$.
It is left as an exercise for the reader to verify that, for any sequence $\sigma$ among those 44 sequences, there exists a sequence $\pi$ for which all of the following hold: 
(i) $\pi$ is one of the sequences (of order $3$) that appears in the first column of Table \ref{long table};
(ii) $\sigma$ is a subsequence of $\pi$; and
(iii) a real symmetric matrix is referenced in the second column of the row of Table \ref{long table} in which $\pi$ appears
(that $\pi$ may be chosen so that that matrix is real is relevant to Section \ref{s: real symmetric}, where we will reference this proof). Thus, those $44$ sequences are not forbidden.
\end{proof}

The sequences $\tt A^+S^*S^+$ and \ $\tt S^+S^*S^+$ are forbidden.

\begin{thm}\label{A+S*S+ and S+S*S+}
Neither $\tt A^+S^*S^+$ nor \ $\tt S^+S^*S^+$
occurs as a subsequence of the sepr-sequence of any Hermitian matrix.
\end{thm}

\begin{proof}
Suppose on the contrary that there exists a Hermitian matrix $B$ such that
$\sepr(B) := t_1t_2 \cdots t_n$ and
$t_{k-1}t_kt_{k+1} \in
\{ \tt A^+S^*S^+, S^+S^*S^+\}$, 
for some integer $k$ with $2 \leq k \leq n-1$.
Let $m$ be the minimal integer with  
$k+1 \leq m \leq n$ for which 
$B$ has an $m \times m$ principal submatrix $C$ such that $\sepr(C) := c_1c_2 \cdots c_m$ and $c_{k} =  \tt S^*$. 
By the Inheritance Theorem,
$c_{k-1}, c_{k+1} \in \{\tt A^+,N,S^+\}$.
Then, as $\tt S^*N$ and $\tt NS^*$ are forbidden (by Theorem \ref{main result for order 2}), 
$c_{k-1}, c_{k+1} \in \{\tt A^+,S^+\}$; 
moreover, as $\tt A^+S^*A^+$ and $\tt S^+S^*A^+$ are forbidden (by Theorem \ref{sepr order-3 summary}), 
$c_{k-1}c_kc_{k+1} \in
\{ \tt A^+S^*S^+, S^+S^*S^+\}$.
Thus, $c_{k+1} = \tt S^+$.
Then, as no matrix has an sepr-sequence whose 
last term is $\tt S^+$, the order of $C$ is at least $k+2$; that is, $m \geq k+2$. Thus, $m \geq 4$.

Let $C'$ be a nonsingular $(k+1) \times (k+1)$ principal submatrix of $C$ and 
$\sepr(C')=c'_1c'_2 \cdots c'_{k+1}$
($C'$ exists because $c_{k+1} = \tt S^+$).
As $C'$ is nonsingular and of order $k+1$, 
and because $c_{k+1} = \tt S^+$, 
$c'_{k+1} = \tt A^+$.

\noindent
{\bf Claim 1}: $c'_{k} \neq \tt N$.

\noindent
We will now establish Claim 1.
By the Inheritance Theorem, 
$c'_{k-1} \in \{\tt A^+,N,S^+\}$.
Then, as
$c'_{k-1}c'_k c'_{k+1}=c'_{k-1}c'_k\tt A^+$,
Theorem \ref{sepr order-3 summary} and 
the $\tt NN$ Theorem imply that $c'_{k} \neq \tt N$. 
Thus, Claim 1 is true.

Because of Observation \ref{permutation}, we may assume, without loss of generality, that the rows and columns of $C$ are ordered in such a way that 
$C'=C[\{1,2, \dots , k+1\}]$.
By Claim 1, there exists a set 
$\alpha \subseteq \{1,2, \dots, k+1\}$ such that 
$|\alpha|=k$ and $C'[\alpha]$ is nonsingular.
Then, as $C'=C[\{1,2, \dots , k+1\}]$,
Observation \ref{permutation} implies that we may assume, without loss of generality, that the rows and columns of $C$ are ordered in such a way that
$\alpha = \{1,2,\dots,k\}$. 
%%%%%%%%%%%%%%%%
%%%%%     WARNING    %%%
%%%%%%%%%%%%%%%%
%%%%% {\red In the previous sentence, it may seem as if you meant to say ``rows and columns of $C'$'' instead of ``rows and columns of $C$'', but that is not the case.}
%%%%%%%%%%%%%%%%
%%%%%%%%%%%%%%%%
Then, as $C'=C[\{1,2, \dots , k+1\}]$,
$C[\alpha] = C'[\alpha]$.
Thus, $\det(C[\alpha]) \neq 0$.
As $c_k = \tt S^*$, there exists a set $\beta \subseteq \{1,2,\dots,m\}$ such that $|\beta| = k$, $C[\beta]$ is nonsingular and the sign of $\det(C[\beta])$ is the opposite of the sign of $\det(C[\alpha])$.
Recall that $\alpha = \{1,2,\dots,k\}$.
Thus, $1 \in \alpha$.
As $|\alpha|=|\beta|$, 
and because $\alpha \neq \beta$,
$\alpha \not\subseteq \beta$.
Observe that, without loss of generality, we may assume that the rows and columns of $C$ that are indexed by $\alpha$ are ordered in such a way that 
$1 \notin \alpha \cap \beta$.
Then, as $1 \in \alpha$, $1 \notin \beta$.
Thus, $C[\beta]$ is a principal submatrix of the $(m-1)\times (m-1)$ submatrix $Y:=C[\{2,3,\dots,m\}]$.
As $\alpha = \{1,2,\dots,k\}$, 
and because $k \leq m-2$,
$C[\alpha]$ is a principal submatrix of 
$X:=C[\{1,2,\dots,m-1\}]$.
Let $\sepr(X)=x_1x_2 \cdots x_{m-1}$ and 
$\sepr(Y)=y_1y_2 \cdots y_{m-1}$.
Observe that, since $m \geq k+2$, 
$X$ and $Y$ are matrices of order at least $k+1$.
By the Inheritance Theorem, 
\[
x_{k-1}, y_{k-1}, x_{k+1}, y_{k+1}, \in 
\{\tt A^+,N,S^+\}.
\]

\noindent
{\bf Claim 2}: $x_k, y_k \notin \{\tt A^*, S^*\}$.

\noindent
We will now establish Claim 2.
As $\tt U^+A^*V^+$ is forbidden for all
$\tt U, V \in \{A,S\}$ 
(by Theorem \ref{sepr order-3 summary}),
and because $\tt A^*N$ and $\tt NA^*$ are forbidden (by Theorem \ref{main result for order 2}), 
$x_k \neq \tt A^*$ and $y_k \neq \tt A^*$.
As $X$ and $Y$ are principal submatrices of 
(not only $C$ but also) $B$,
and because their order (i.e., $m-1$) 
is at least $k+1$ and less than $m$,
the minimality of $m$ implies that 
$x_k \neq \tt S^*$ and $y_k \neq \tt S^*$.
Thus, Claim 2 is true.

Let $Z=C[\{1,2, \dots, m\} \setminus \{1,m\}]$ and
$\sepr(Z)=z_1z_2 \cdots z_{m-2}$.
Observe that, since $m \geq k+2$, 
the order of $Z$ (i.e., $m-2$) is at least $k$.

\noindent
{\bf Claim 3}: $z_{k} = \tt N$.

\noindent
We will now establish Claim 3.
Let $M$ be a $k \times k$ principal submatrix of $Z$.
We will now prove that $\det(M)=0$.
By Claim 2, all the nonzero order-$k$ principal minors of $X$ have the same sign; moreover,
all the nonzero order-$k$ principal minors of $Y$ have the same sign.
Recall that $C[\alpha]$ and $C[\beta]$ are nonsingular and $k \times k$, and that the former is a principal submatrix of $X$, and that the latter is a principal submatrix of $Y$.
Thus, all the nonzero order-$k$ principal minors of $X$ have the same sign as $\det(C[\alpha])$;
moreover, all the nonzero order-$k$ principal minors of $Y$ have the same sign as $\det(C[\beta])$.
As $M$ is a $k \times k$ principal submatrix of both $X$ and $Y$, $\det(M)$ is an order-$k$ principal minor of both $X$ and $Y$.
Thus, if it were the case that $\det(M) \neq 0$, 
then we would have
$\det(M)\det(C[\alpha])>0$ and 
$\det(M)\det(C[\beta])>0$, which would imply that 
$\det(C[\alpha])$ and $\det(C[\beta])$ have the same sign, leading to a contradiction.
Thus, $\det(M) =0$, as desired.
Then, as $M$ was arbitrary, Claim 3 is true.

Recall that $Z=C[\{1,2, \dots, m\}\setminus \{1,m\}]$.
As $m \geq k+2$, $k+1 \leq m-1$,
implying that the $k \times k$ submatrix 
$C[\{2,3,\dots,k+1\}]$ is a 
principal submatrix of $Z$.
Then, as $z_{k} = \tt N$ (by Claim 3),
$C[\{2,3,\dots,k+1\}]$ is singular.
Recall that $C'=C[\{1,2, \dots , k+1\}]$.
Thus, $C'[\{2,3,\dots,k+1\}]=C[\{2,3,\dots,k+1\}]$,
which is singular.
Thus, $C'$ has a singular $k \times k$ principal submatrix.
Then, as $c'_{k} \neq \tt N$ (by Claim 1),
$\uepr(c'_{k}) = \tt S$.
Recall that
$c'_{k-1} \in \{\tt A^+,N, S^+\}$ and 
$c'_{k+1} = \tt A^+$.
Thus, either 
$\uepr(c'_{k-1}c'_{k}c'_{k+1}) = \tt NSA$ 
(which is forbidden by Theorem \ref{NNA NNS NSA})
or 
$c'_{k-1}c'_{k}c'_{k+1}$ is a sequence that is forbidden by Theorem \ref{sepr order-3 summary}.
Thus, we have reached a contradiction, implying that
$\tt A^+S^*S^+$ and \ $\tt S^+S^*S^+$ are forbidden.
\end{proof}

The proof of the following theorem, 
which states that the sequences
$\tt A^-S^*S^-$ and \ $\tt S^-S^*S^-$ are 
forbidden, was obtained simply 
by replacing each `+' superscript in the proof of Theorem \ref{A+S*S+ and S+S*S+} with a `-' superscript; it has been included here in the interest of completeness.

\begin{thm}\label{A-S*S- and S-S*S-}
Neither $\tt A^-S^*S^-$ nor \ $\tt S^-S^*S^-$
occurs as a subsequence of the sepr-sequence of any Hermitian matrix.
\end{thm}

\begin{proof}
Suppose on the contrary that there exists a Hermitian matrix $B$ such that
$\sepr(B) := t_1t_2 \cdots t_n$ and
$t_{k-1}t_kt_{k+1} \in
\{ \tt A^-S^*S^-, S^-S^*S^-\}$, 
for some integer $k$ with $2 \leq k \leq n-1$.
Let $m$ be the minimal integer with  
$k+1 \leq m \leq n$ for which 
$B$ has an $m \times m$ principal submatrix $C$ such that $\sepr(C) := c_1c_2 \cdots c_m$ and $c_{k} =  \tt S^*$. 
By the Inheritance Theorem,
$c_{k-1}, c_{k+1} \in \{\tt A^-,N,S^-\}$.
Then, as $\tt S^*N$ and $\tt NS^*$ are forbidden (by Theorem \ref{main result for order 2}), 
$c_{k-1}, c_{k+1} \in \{\tt A^-,S^-\}$; 
moreover, as $\tt A^-S^*A^-$ and $\tt S^-S^*A^-$ are forbidden (by Theorem \ref{sepr order-3 summary}), 
$c_{k-1}c_kc_{k+1} \in
\{ \tt A^-S^*S^-, S^-S^*S^-\}$.
Thus, $c_{k+1} = \tt S^-$.
Then, as no matrix has an sepr-sequence whose 
last term is $\tt S^-$, the order of $C$ is at least $k+2$; that is, $m \geq k+2$. Thus, $m \geq 4$.

Let $C'$ be a nonsingular $(k+1) \times (k+1)$ principal submatrix of $C$ and 
$\sepr(C')=c'_1c'_2 \cdots c'_{k+1}$
($C'$ exists because $c_{k+1} = \tt S^-$).
As $C'$ is nonsingular and of order $k+1$, 
and because $c_{k+1} = \tt S^-$, 
$c'_{k+1} = \tt A^-$.

\noindent
{\bf Claim 1}: $c'_{k} \neq \tt N$.

\noindent
We will now establish Claim 1.
By the Inheritance Theorem, 
$c'_{k-1} \in \{\tt A^-,N,S^-\}$.
Then, as
$c'_{k-1}c'_k c'_{k+1}=c'_{k-1}c'_k\tt A^-$,
Theorem \ref{sepr order-3 summary} and 
the $\tt NN$ Theorem imply that $c'_{k} \neq \tt N$. 
Thus, Claim 1 is true.

Because of Observation \ref{permutation}, we may assume, without loss of generality, that the rows and columns of $C$ are ordered in such a way that 
$C'=C[\{1,2, \dots , k+1\}]$.
By Claim 1, there exists a set 
$\alpha \subseteq \{1,2, \dots, k+1\}$ such that 
$|\alpha|=k$ and $C'[\alpha]$ is nonsingular.
Then, as $C'=C[\{1,2, \dots , k+1\}]$,
Observation \ref{permutation} implies that we may assume, without loss of generality, that the rows and columns of $C$ are ordered in such a way that
$\alpha = \{1,2,\dots,k\}$. 
%%%%%%%%%%%%%%%%
%%%%%     WARNING    %%%
%%%%%%%%%%%%%%%%
%%%%% {\red In the previous sentence, it may seem as if you meant to say ``rows and columns of $C'$'' instead of ``rows and columns of $C$'', but that is not the case.}
%%%%%%%%%%%%%%%%
%%%%%%%%%%%%%%%%
Then, as $C'=C[\{1,2, \dots , k+1\}]$,
$C[\alpha] = C'[\alpha]$.
Thus, $\det(C[\alpha]) \neq 0$.
As $c_k = \tt S^*$, there exists a set $\beta \subseteq \{1,2,\dots,m\}$ such that $|\beta| = k$, $C[\beta]$ is nonsingular and the sign of $\det(C[\beta])$ is the opposite of the sign of $\det(C[\alpha])$.
Recall that $\alpha = \{1,2,\dots,k\}$.
Thus, $1 \in \alpha$.
As $|\alpha|=|\beta|$, 
and because $\alpha \neq \beta$,
$\alpha \not\subseteq \beta$.
Observe that, without loss of generality, we may assume that the rows and columns of $C$ that are indexed by $\alpha$ are ordered in such a way that 
$1 \notin \alpha \cap \beta$.
Then, as $1 \in \alpha$, $1 \notin \beta$.
Thus, $C[\beta]$ is a principal submatrix of the $(m-1)\times (m-1)$ submatrix $Y:=C[\{2,3,\dots,m\}]$.
As $\alpha = \{1,2,\dots,k\}$, 
and because $k \leq m-2$,
$C[\alpha]$ is a principal submatrix of 
$X:=C[\{1,2,\dots,m-1\}]$.
Let $\sepr(X)=x_1x_2 \cdots x_{m-1}$ and 
$\sepr(Y)=y_1y_2 \cdots y_{m-1}$.
Observe that, since $m \geq k+2$, 
$X$ and $Y$ are matrices of order at least $k+1$.
By the Inheritance Theorem, 
\[
x_{k-1}, y_{k-1}, x_{k+1}, y_{k+1}, \in 
\{\tt A^-,N,S^-\}.
\]

\noindent
{\bf Claim 2}: $x_k, y_k \notin \{\tt A^*, S^*\}$.

\noindent
We will now establish Claim 2.
As $\tt U^-A^*V^-$ is forbidden for all
$\tt U, V \in \{A,S\}$ 
(by Theorem \ref{sepr order-3 summary}),
and because $\tt A^*N$ and $\tt NA^*$ are forbidden (by Theorem \ref{main result for order 2}), 
$x_k \neq \tt A^*$ and $y_k \neq \tt A^*$.
As $X$ and $Y$ are principal submatrices of 
(not only $C$ but also) $B$,
and because their order (i.e., $m-1$) 
is at least $k+1$ and less than $m$,
the minimality of $m$ implies that 
$x_k \neq \tt S^*$ and $y_k \neq \tt S^*$.
Thus, Claim 2 is true.

Let $Z=C[\{1,2, \dots, m\} \setminus \{1,m\}]$ and
$\sepr(Z)=z_1z_2 \cdots z_{m-2}$.
Observe that, since $m \geq k+2$, 
the order of $Z$ (i.e., $m-2$) is at least $k$.

\noindent
{\bf Claim 3}: $z_{k} = \tt N$.

\noindent
We will now establish Claim 3.
Let $M$ be a $k \times k$ principal submatrix of $Z$.
We will now prove that $\det(M)=0$.
By Claim 2, all the nonzero order-$k$ principal minors of $X$ have the same sign; moreover,
all the nonzero order-$k$ principal minors of $Y$ have the same sign.
Recall that $C[\alpha]$ and $C[\beta]$ are nonsingular and $k \times k$, and that the former is a principal submatrix of $X$, and that the latter is a principal submatrix of $Y$.
Thus, all the nonzero order-$k$ principal minors of $X$ have the same sign as $\det(C[\alpha])$;
moreover, all the nonzero order-$k$ principal minors of $Y$ have the same sign as $\det(C[\beta])$.
As $M$ is a $k \times k$ principal submatrix of both $X$ and $Y$, $\det(M)$ is an order-$k$ principal minor of both $X$ and $Y$.
Thus, if it were the case that $\det(M) \neq 0$, 
then we would have
$\det(M)\det(C[\alpha])>0$ and 
$\det(M)\det(C[\beta])>0$, which would imply that 
$\det(C[\alpha])$ and $\det(C[\beta])$ have the same sign, leading to a contradiction.
Thus, $\det(M) =0$, as desired.
Then, as $M$ was arbitrary, Claim 3 is true.

Recall that $Z=C[\{1,2, \dots, m\}\setminus \{1,m\}]$.
As $m \geq k+2$, $k+1 \leq m-1$,
implying that the $k \times k$ submatrix 
$C[\{2,3,\dots,k+1\}]$ is a 
principal submatrix of $Z$.
Then, as $z_{k} = \tt N$ (by Claim 3),
$C[\{2,3,\dots,k+1\}]$ is singular.
Recall that $C'=C[\{1,2, \dots , k+1\}]$.
Thus, $C'[\{2,3,\dots,k+1\}]=C[\{2,3,\dots,k+1\}]$,
which is singular.
Thus, $C'$ has a singular $k \times k$ principal submatrix.
Then, as $c'_{k} \neq \tt N$ (by Claim 1),
$\uepr(c'_{k}) = \tt S$.
Recall that
$c'_{k-1} \in \{\tt A^-,N, S^-\}$ and 
$c'_{k+1} = \tt A^-$.
Thus, either 
$\uepr(c'_{k-1}c'_{k}c'_{k+1}) = \tt NSA$ 
(which is forbidden by Theorem \ref{NNA NNS NSA})
or 
$c'_{k-1}c'_{k}c'_{k+1}$ is a sequence that is forbidden by Theorem \ref{sepr order-3 summary}.
Thus, we have reached a contradiction, implying that
$\tt A^-S^*S^-$ and \ $\tt S^-S^*S^-$ are forbidden.
\end{proof}

Theorem \ref{main result for order 2} and the next theorem, together, provide the answer to Question \ref{question-Hermitian}.

\begin{thm}\label{main result for order 3}
A sequence $\sigma$ from 
$\{\tt A^*,A^+,A^-,N,S^*,S^+,S^-\}$ of order $3$ does not occur as a subsequence of the sepr-sequence of any Hermitian matrix
if and only if 
one of the following statements holds:
\begin{enumerate}
\item $\sigma \in \{\tt A^+XA^+,A^-XA^-,S^+XA^+,S^-XA^-\}$, for some 
$\tt X \in \{\tt A^*,N,S^*,S^+,S^-\}$. \label{AXA and SXA}
\item $\sigma \in \{\tt A^+YS^+,A^-YS^-,S^+YS^+,S^-YS^-\}$, for some
$\tt Y \in \{\tt A^*,N,S^*\}$. \label{AYS and SYS}
\item  $\tt A^*N$ or \ $\tt NA^*$ or \ $\tt S^*N$ or \ $\tt NS^*$ is a subsequence of $\sigma$. \label{star N and N star}
\item $\uepr(\sigma) \in \{\tt NNA, NNS, NSA\}$. \label{NNA and NNS and NSA}
\end{enumerate}
\end{thm}

\begin{proof}
Let $\T$ be the set whose elements are all
the sequences $\sigma$ from  
$\{\tt A^*,A^+,A^-,N,S^*,S^+,S^-\}$ of order $3$
for which
statement (\ref{AXA and SXA}) or (\ref{AYS and SYS}) or (\ref{star N and N star}) or (\ref{NNA and NNS and NSA}) holds. 
Thus, $\sigma \in \T$ if and only if
one of those four statements holds.
Let $\F$ be the set whose elements are 
all the sequences from  
$\{\tt A^*,A^+,A^-,N,S^*,S^+,S^-\}$ of order $3$ 
that are forbidden.
Thus, it suffices to prove that 
$\T = \F$.

We will now prove that $\T \subseteq \F$.
To do that, it suffices to prove that 
if statement (\ref{AXA and SXA}) or (\ref{AYS and SYS}) or (\ref{star N and N star}) or (\ref{NNA and NNS and NSA}) holds, then $\sigma \in \F$.
%%%%%%%
If statement (\ref{AXA and SXA})  holds, 
then, by Theorem \ref{sepr order-3 summary},
$\sigma \in \F$.
%%%%%%%
If statement (\ref{AYS and SYS}) holds, then
Theorem \ref{sepr order-3 summary} and 
Theorem \ref{A+S*S+ and S+S*S+} and
Theorem \ref{A-S*S- and S-S*S-} imply that
$\sigma \in \F$.
%%%%%%%
If statement (\ref{star N and N star}) holds, then Theorem \ref{main result for order 2} implies that 
$\sigma \in \F$.
%%%%%%%
If statement (\ref{NNA and NNS and NSA}) holds,
then, by Theorem \ref{NNA NNS NSA},
$\sigma \in \F$.
%%%%%%%
It follows, then, that $\T \subseteq \F$, as desired.

To conclude the proof, 
it suffices to prove that $|\T|=|\F|$.
It is left as an exercise for the reader to 
verify that each of the following 92 sequences is an element of the set $\T$:
$\tt A^+ A^* A^+$,
$\tt A^- A^* A^-$,
$\tt A^* A^* N$,
$\tt A^+ A^* N$,
$\tt A^- A^* N$,
$\tt A^+ A^* S^+$,
$\tt A^- A^* S^-$,
$\tt A^* N A^*$,
$\tt A^* N A^+$,
$\tt A^* N A^-$,
$\tt A^+ N A^*$,
$\tt A^+ N A^+$,
$\tt A^- N A^*$,
$\tt A^- N A^-$,
$\tt A^* N N$,
$\tt A^* N S^*$,
$\tt A^* N S^+$,
$\tt A^* N S^-$,
$\tt A^+ N S^*$,
$\tt A^+ N S^+$,
$\tt A^- N S^*$,
$\tt A^- N S^-$,
$\tt A^+ S^* A^+$,
$\tt A^+ S^+ A^+$,
$\tt A^+ S^- A^+$,
$\tt A^- S^* A^-$,
$\tt A^- S^+ A^-$,
$\tt A^- S^- A^-$,
$\tt A^* S^* N$,
$\tt A^+ S^* N$,
$\tt A^- S^* N$,
$\tt A^+ S^* S^+$,
$\tt A^- S^* S^-$,
$\tt N A^* A^*$,
$\tt N A^* A^+$,
$\tt N A^* A^-$,
$\tt N A^* N$,
$\tt N A^* S^*$,
$\tt N A^* S^+$,
$\tt N A^* S^-$,
$\tt N N A^*$,
$\tt N N A^+$,
$\tt N N A^-$,
$\tt N N S^*$,
$\tt N N S^+$,
$\tt N N S^-$,
$\tt N S^* A^*$,
$\tt N S^* A^+$,
$\tt N S^* A^-$,
$\tt N S^+ A^*$,
$\tt N S^+ A^+$,
$\tt N S^+ A^-$,
$\tt N S^- A^*$,
$\tt N S^- A^+$,
$\tt N S^- A^-$,
$\tt N S^* N$,
$\tt N S^* S^*$,
$\tt N S^* S^+$,
$\tt N S^* S^-$,
$\tt S^+ A^* A^+$,
$\tt S^- A^* A^-$,
$\tt S^* A^* N$,
$\tt S^+ A^* N$,
$\tt S^- A^* N$,
$\tt S^+ A^* S^+$,
$\tt S^- A^* S^-$,
$\tt S^* N A^*$,
$\tt S^* N A^+$,
$\tt S^* N A^-$,
$\tt S^+ N A^*$,
$\tt S^+ N A^+$,
$\tt S^- N A^*$,
$\tt S^- N A^-$,
$\tt S^* N N$,
$\tt S^* N S^*$,
$\tt S^* N S^+$,
$\tt S^* N S^-$,
$\tt S^+ N S^*$,
$\tt S^+ N S^+$,
$\tt S^- N S^*$,
$\tt S^- N S^-$,
$\tt S^+ S^* A^+$,
$\tt S^+ S^+ A^+$,
$\tt S^+ S^- A^+$,
$\tt S^- S^* A^-$,
$\tt S^- S^+ A^-$,
$\tt S^- S^- A^-$,
$\tt S^* S^* N$,
$\tt S^+ S^* N$,
$\tt S^- S^* N$,
$\tt S^+ S^* S^+$ and
$\tt S^- S^* S^-$.
Thus, $|\T| \geq 92$.
Observe that there are $7^3=343$ sequences from  
$\{\tt A^*,A^+,A^-,N,S^*,S^+,S^-\}$ of order $3$.
The 65 sequences listed in \cite[Table 2]{XMR-sepr} and the 186 sequences listed in Table \ref{long table} are among those 343 sequences;
furthermore, 
as no sequence appears in both of those tables, 
at least $251$ ($65+186$) of those 343 sequences are not forbidden. 
Thus, $|\F| \leq 343-251 = 92$.
Then, as 
$\T \subseteq \F$, and because $|\T| \geq 92$, 
$|\T| = |\F|$, as desired.
\end{proof}

\begin{longtable}[c]{| l | c | l |}
\caption{The first column of this table contains all the sequences of order $3$ not appearing in \cite[Table 2]{XMR-sepr} that are \textit{not} forbidden
(i.e., that occur as a subsequence of the sepr-sequence of some Hermitian matrix).
For each sequence $\sigma$ in the first column of the table, the information that appears in the row that contains $\sigma$ provides sufficient information to establish that $\sigma$ is not forbidden;
in particular, if a matrix $M$ is referenced in that row (in the second column), then $\sigma$ is a subsequence of $\sepr(M)$ and $M$ is defined in the statement of the result that is stated in that row (in the third column).
} \label{long table}\\
%%%%%%%%%%%%%%%%%%%%
%%%%%%%%%%%%%%%%%%%%
    \hline
       Sequence & Hermitian matrix & Result \\ \hline  
        $\tt A^* A^* A^*$ & diag(1,1,1,-1) & ~ \\ \hline  % 1 \\ \hline
        $\tt A^* A^+ A^*$ & $\FiveOneA$ & Lemma \ref{FiveOne} (\ref{FiveOneA}) \\ \hline  % 4 \\ \hline
        $\tt A^* A^+ A^+$ & $(\VierOneA)^{-1}$ & Lemma \ref{VierOne} (\ref{VierOneA}) \\ \hline  % 5 \\ \hline
        $\tt A^* A^+ A^-$ & $\VierFiveB$ & Lemma \ref{VierFive} (\ref{VierFiveB}) \\ \hline  % 6 \\ \hline
        $\tt A^* A^- A^*$ & $-\FiveOneA$ & Lemma \ref{FiveOne} (\ref{FiveOneA}) \\ \hline  % 7 \\ \hline
        $\tt A^+ A^* A^*$ & $\VierOneA$ & Lemma \ref{VierOne} (\ref{VierOneA}) \\ \hline  % 10 \\ \hline
        $\tt A^+ A^+ A^*$ & $\VierOneB$ & Lemma \ref{VierOne} (\ref{VierOneB}) \\ \hline  % 13 \\ \hline
        $\tt A^+ A^- A^*$ & $\VierOneC$ & Lemma \ref{VierOne} (\ref{VierOneC}) \\ \hline  % 16 \\ \hline
        $\tt A^- A^* A^*$ & $-\VierOneA$ & Lemma \ref{VierOne} (\ref{VierOneA}) \\ \hline  % 19 \\ \hline
        $\tt A^- A^+ A^*$ & $-\VierOneB$ & Lemma \ref{VierOne} (\ref{VierOneB}) \\ \hline  % 22 \\ \hline
        $\tt A^- A^- A^*$ & $-\VierOneC$ & Lemma \ref{VierOne} (\ref{VierOneC}) \\ \hline  % 25 \\ \hline
        $\tt A^* A^+ N$ & $\VierOneD$ & Lemma \ref{VierOne} (\ref{VierOneD}) \\ \hline  % 29 \\ \hline
        $\tt A^* A^* S^*$ & $(\VierFourA)^{-1}$ & Lemma \ref{VierFour} (\ref{VierFourA}) \\ \hline  % 37 \\ \hline
        $\tt A^* A^* S^+$ & $\FiveOneB$ & Lemma \ref{FiveOne} (\ref{FiveOneB}) \\ \hline  % 38 \\ \hline
        $\tt A^* A^* S^-$ & $-\FiveOneB$ & Lemma \ref{FiveOne} (\ref{FiveOneB}) \\ \hline  % 39 \\ \hline
        $\tt A^* A^+ S^*$ & $\FiveTwoA$ & Lemma \ref{FiveTwo} (\ref{FiveTwoA}) \\ \hline  % 40 \\ \hline
        $\tt A^* A^+ S^+$ & $\FiveTwoB$ & Lemma \ref{FiveTwo} (\ref{FiveTwoB}) \\ \hline  % 41 \\ \hline
        $\tt A^* A^+ S^-$ & $\FiveTwoC$ & Lemma \ref{FiveTwo} (\ref{FiveTwoC}) \\ \hline  % 42 \\ \hline
        $\tt A^* A^- S^*$ & $-\FiveTwoA$ & Lemma \ref{FiveTwo} (\ref{FiveTwoA}) \\ \hline  % 43 \\ \hline
        $\tt A^* A^- S^+$ & $-\FiveTwoB$ & Lemma \ref{FiveTwo} (\ref{FiveTwoB}) \\ \hline  % 44 \\ \hline
        $\tt A^* A^- S^-$ & $-\FiveTwoC$ & Lemma \ref{FiveTwo} (\ref{FiveTwoC}) \\ \hline  % 45 \\ \hline
        $\tt A^+ A^* S^*$ & $(\VierFiveA)^{-1}$ & Lemma \ref{VierFive} (\ref{VierFiveA}) \\ \hline  % 46 \\ \hline
        $\tt A^+ A^* S^-$ & $\VierOneE$ & Lemma \ref{VierOne} (\ref{VierOneE}) \\ \hline  % 48 \\ \hline
        $\tt A^+ A^+ S^*$ & $-(\VierFourC)^{-1}$ & Lemma \ref{VierFour} (\ref{VierFourC}) \\ \hline  % 49 \\ \hline
        $\tt A^+ A^+ S^+$ & $-(\VierFourF)^{-1}$ & Lemma \ref{VierFour} (\ref{VierFourF}) \\ \hline  % 50 \\ \hline
        $\tt A^+ A^+ S^-$ & $(\VierFourH)^{-1}$ & Lemma \ref{VierFour} (\ref{VierFourH}) \\ \hline  % 51 \\ \hline
        $\tt A^+ A^- S^*$ & $(\VierFourB)^{-1}$ & Lemma \ref{VierFour} (\ref{VierFourB}) \\ \hline  % 52 \\ \hline
        $\tt A^+ A^- S^+$ & $(\VierFourE)^{-1}$ & Lemma \ref{VierFour} (\ref{VierFourE}) \\ \hline  % 53 \\ \hline
        $\tt A^+ A^- S^-$ & $-(\VierFourG)^{-1}$ & Lemma \ref{VierFour} (\ref{VierFourG}) \\ \hline  % 54 \\ \hline
        $\tt A^- A^* S^*$ & $(\VierFiveB)^{-1}$ & Lemma \ref{VierFive} (\ref{VierFiveB}) \\ \hline  % 55 \\ \hline
        $\tt A^- A^* S^+$ & $\SixOneA$ & Lemma \ref{SixOne} (\ref{SixOneA}) \\ \hline  % 56 \\ \hline
        $\tt A^- A^+ S^*$ & $(\VierFourC)^{-1}$ & Lemma \ref{VierFour} (\ref{VierFourC}) \\ \hline  % 58 \\ \hline
        $\tt A^- A^+ S^+$ & $-(\VierFourH)^{-1}$ & Lemma \ref{VierFour} (\ref{VierFourH}) \\ \hline  % 59 \\ \hline
        $\tt A^- A^+ S^-$ & $(\VierFourF)^{-1}$ & Lemma \ref{VierFour} (\ref{VierFourF}) \\ \hline  % 60 \\ \hline
        $\tt A^- A^- S^*$ & $-(\VierFourB)^{-1}$ & Lemma \ref{VierFour} (\ref{VierFourB}) \\ \hline  % 61 \\ \hline
        $\tt A^- A^- S^+$ & $(\VierFourG)^{-1}$ & Lemma \ref{VierFour} (\ref{VierFourG}) \\ \hline  % 62 \\ \hline
        $\tt A^- A^- S^-$ & $-(\VierFourE)^{-1}$ & Lemma \ref{VierFour} (\ref{VierFourE}) \\ \hline  % 63 \\ \hline
        $\tt A^+ N S^-$ & $-\NSFrealB$ & Lemma \ref{NSFreal} (\ref{NSFrealB})  \\ \hline  % 81 \\ \hline
        $\tt A^- N S^+$ & $\NSFrealB$ & Lemma \ref{NSFreal} (\ref{NSFrealB})  \\ \hline  % 83 \\ \hline
        $\tt A^* S^* A^*$ & $\VierOneF$ & Lemma \ref{VierOne} (\ref{VierOneF}) \\ \hline  % 85 \\ \hline
        $\tt A^* S^* A^+$ & $\VierTwoA$ & Lemma \ref{VierTwo} (\ref{VierTwoA})  \\ \hline  % 86 \\ \hline
        $\tt A^* S^* A^-$ & $-\VierTwoA$ & Lemma \ref{VierTwo} (\ref{VierTwoA})  \\ \hline  % 87 \\ \hline
        $\tt A^* S^+ A^*$ & $\FiveOneB$ & Lemma \ref{FiveOne} (\ref{FiveOneB}) \\ \hline  % 88 \\ \hline
        $\tt A^* S^+ A^+$ & $\SixOneB$  & Lemma \ref{SixOne} (\ref{SixOneB}) \\ \hline  % 89 \\ \hline
        $\tt A^* S^+ A^-$ & $-(\VierThreeC)^{-1}$ & Lemma \ref{VierThree} (\ref{VierThreeC}) \\ \hline  % 90 \\ \hline
        $\tt A^* S^- A^*$ & $-\FiveOneB$ & Lemma \ref{FiveOne} (\ref{FiveOneB}) \\ \hline  % 91 \\ \hline
        $\tt A^+ S^* A^*$ & $\VierOneG$ & Lemma \ref{VierOne} (\ref{VierOneG}) \\ \hline  % 94 \\ \hline
        $\tt A^+ S^+ A^*$ & $\VierOneH$ & Lemma \ref{VierOne} (\ref{VierOneH}) \\ \hline  % 97 \\ \hline
        $\tt A^+ S^- A^*$ & $\VierOneJ$ & Lemma \ref{VierOne} (\ref{VierOneJ}) \\ \hline  % 100 \\ \hline
        $\tt A^- S^* A^*$ & $-\VierOneG$ & Lemma \ref{VierOne} (\ref{VierOneG}) \\ \hline  % 103 \\ \hline
        $\tt A^- S^+ A^*$ & $-\VierOneH$ & Lemma \ref{VierOne} (\ref{VierOneH}) \\ \hline  % 106 \\ \hline
        $\tt A^- S^- A^*$ & $-\VierOneJ$ & Lemma \ref{VierOne} (\ref{VierOneJ}) \\ \hline  % 109 \\ \hline
        $\tt A^* S^+ N$ & $-\VierOneE$ & Lemma \ref{VierOne} (\ref{VierOneE}) \\ \hline  % 113 \\ \hline
        $\tt A^* S^* S^*$ & $(\VierFourK)^{-1}$ & Lemma \ref{VierFour} (\ref{VierFourK}) \\ \hline  % 121 \\ \hline
        $\tt A^* S^* S^+$ & ~ & Lemma \ref{ASS etc} \\ \hline  % 122 \\ \hline
        $\tt A^* S^* S^-$ & ~ & Lemma \ref{ASS etc} \\ \hline  % 123 \\ \hline
        $\tt A^* S^+ S^*$ & $\FiveTwoD$ & Lemma \ref{FiveTwo} (\ref{FiveTwoD}) \\ \hline  % 124 \\ \hline
        $\tt A^* S^+ S^+$ & $\FiveTwoE$ & Lemma \ref{FiveTwo} (\ref{FiveTwoE}) \\ \hline  % 125 \\ \hline
        $\tt A^* S^+ S^-$ & $\FiveTwoF$ & Lemma \ref{FiveTwo} (\ref{FiveTwoF}) \\ \hline  % 126 \\ \hline
        $\tt A^* S^- S^*$ & $-\FiveTwoD$ & Lemma \ref{FiveTwo} (\ref{FiveTwoD}) \\ \hline  % 127 \\ \hline
        $\tt A^* S^- S^+$ & ~ & Lemma \ref{ASS etc} \\ \hline  % 128 \\ \hline
        $\tt A^* S^- S^-$ & ~ & Lemma \ref{ASS etc} \\ \hline  % 129 \\ \hline
        $\tt A^+ S^* S^*$ & $(\VierFiveD)^{-1}$ & Lemma \ref{VierFive} (\ref{VierFiveD}) \\ \hline  % 130 \\ \hline
        $\tt A^+ S^* S^-$ & ~ & Lemma \ref{ASS etc} \\ \hline  % 132 \\ \hline
        $\tt A^+ S^+ S^*$ & $-(\VierFourM)^{-1}$ & Lemma \ref{VierFour} (\ref{VierFourM}) \\ \hline  % 133 \\ \hline
        $\tt A^+ S^+ S^+$ & ~ & Lemma \ref{ASS etc} \\ \hline  % 134 \\ \hline
        $\tt A^+ S^+ S^-$ & ~ & Lemma \ref{ASS etc} \\ \hline  % 135 \\ \hline
        $\tt A^+ S^- S^*$ & $(\VierFourL)^{-1}$ & Lemma \ref{VierFour} (\ref{VierFourL}) \\ \hline  % 136 \\ \hline
        $\tt A^+ S^- S^+$ & ~ & Lemma \ref{ASS etc} \\ \hline  % 137 \\ \hline
        $\tt A^+ S^- S^-$ & ~ & Lemma \ref{ASS etc} \\ \hline  % 138 \\ \hline
        $\tt A^- S^* S^*$ & $-(\VierFiveD)^{-1}$ & Lemma \ref{VierFive} (\ref{VierFiveD}) \\ \hline  % 139 \\ \hline
        $\tt A^- S^* S^+$ & ~ & Lemma \ref{ASS etc} \\ \hline  % 140 \\ \hline
        $\tt A^- S^+ S^*$ & $(\VierFourM)^{-1}$ & Lemma \ref{VierFour} (\ref{VierFourM}) \\ \hline  % 142 \\ \hline
        $\tt A^- S^+ S^+$ & ~ & Lemma \ref{ASS etc} \\ \hline  % 143 \\ \hline
        $\tt A^- S^+ S^-$ & ~ & Lemma \ref{ASS etc} \\ \hline  % 144 \\ \hline
        $\tt A^- S^- S^*$ & $-(\VierFourL)^{-1}$ & Lemma \ref{VierFour} (\ref{VierFourL}) \\ \hline  % 145 \\ \hline
        $\tt A^- S^- S^+$ & ~ & Lemma \ref{ASS etc} \\ \hline  % 146 \\ \hline
        $\tt A^- S^- S^-$ & ~ & Lemma \ref{ASS etc} \\ \hline  % 147 \\ \hline
%%%%%%%%%%%%
%%%%%%%%%%%%
%%%%%%%%%%%%
%%%%%%%%%%%%
%%%%%%%%%%%%
%%%%%%%%%%%%
        $\tt N A^+ A^*$ & $\NSFcomA$ & Lemma \ref{NSFcom} (\ref{NSFcomA})  \\ \hline  % 4 \\ \hline
        $\tt N A^+ A^+$ & $-(\NSFrealB)^{-1}$ & Lemma \ref{NSFreal} (\ref{NSFrealB}) \\ \hline  % 5 \\ \hline
        $\tt N A^+ A^-$ & $(\NSFrealB)^{-1}$ & Lemma \ref{NSFreal} (\ref{NSFrealB}) \\ \hline  % 6 \\ \hline
        $\tt N A^- A^*$ & $\VierThreeA$ & Lemma \ref{VierThree} (\ref{VierThreeA}) \\ \hline  % 7 \\ \hline
        $\tt N A^+ N$ & $\NSFcomB$ & Lemma \ref{NSFcom} (\ref{NSFcomB})  \\ \hline  % 11 \\ \hline
        $\tt N A^+ S^*$ & $\ComOneA$ & Lemma \ref{Complex} (\ref{ComOneA}) \\ \hline  % 16 \\ \hline
        $\tt N A^+ S^+$ & $\ComOneB$ & Lemma \ref{Complex} (\ref{ComOneB}) \\ \hline  % 17 \\ \hline
        $\tt N A^+ S^-$ & $-\ComOneB$ & Lemma \ref{Complex} (\ref{ComOneB}) \\ \hline  % 18 \\ \hline
        $\tt N A^- S^*$ & $\ComTwoA$ & Lemma \ref{Complex} (\ref{ComTwoA}) \\ \hline  % 19 \\ \hline
        $\tt N A^- S^+$ & $\ComTwoB$ & Lemma \ref{Complex} (\ref{ComTwoB}) \\ \hline  % 20 \\ \hline
        $\tt N A^- S^-$ & $-\ComTwoB$ & Lemma \ref{Complex} (\ref{ComTwoB}) \\ \hline  % 21 \\ \hline
        $\tt N S^+ N$ & $\NSFrealB$ & Lemma \ref{NSFreal} (\ref{NSFrealB})  \\ \hline  % 39 \\ \hline
        $\tt N S^+ S^*$ & $\NSFrealE$ & Lemma \ref{NSFreal} (\ref{NSFrealE})  \\ \hline  % 44 \\ \hline
        $\tt N S^+ S^+$ & $\NSFrealF$ & Lemma \ref{NSFreal} (\ref{NSFrealF})  \\ \hline  % 45 \\ \hline
        $\tt N S^+ S^-$ & $-\NSFrealF$ & Lemma \ref{NSFreal} (\ref{NSFrealF})  \\ \hline  % 46 \\ \hline
        $\tt N S^- S^*$ & $\VierThreeB$ & Lemma \ref{VierThree} (\ref{VierThreeB}) \\ \hline  % 47 \\ \hline
        $\tt N S^- S^+$ & ~ & Lemma \ref{ASS etc} \\ \hline  % 48 \\ \hline
        $\tt N S^- S^-$ & ~ & Lemma \ref{ASS etc} \\ \hline  % 49 \\ \hline
%%%%%%%%%%%%
%%%%%%%%%%%%
%%%%%%%%%%%%
%%%%%%%%%%%%
%%%%%%%%%%%%
%%%%%%%%%%%%
        $\tt S^* A^* A^*$ & $\VierFourA$ & Lemma \ref{VierFour} (\ref{VierFourA}) \\ \hline  % 1 \\ \hline
        $\tt S^* A^* A^+$ & $\VierFiveB$ & Lemma \ref{VierFive} (\ref{VierFiveB}) \\ \hline  % 2 \\ \hline
        $\tt S^* A^* A^-$ & $-\VierFiveB$ & Lemma \ref{VierFive} (\ref{VierFiveB}) \\ \hline  % 3 \\ \hline
        $\tt S^* A^+ A^*$ & $\FiveOneC$ & Lemma \ref{FiveOne} (\ref{FiveOneC}) \\ \hline  % 4 \\ \hline
        $\tt S^* A^+ A^+$ & $\VierFiveD$ & Lemma \ref{VierFive} (\ref{VierFiveD}) \\ \hline  % 5 \\ \hline
        $\tt S^* A^+ A^-$ & $-(\VierOneG)^{-1}$ & Lemma \ref{VierOne} (\ref{VierOneG}) \\ \hline  % 6 \\ \hline
        $\tt S^* A^- A^*$ & $-\FiveOneC$ & Lemma \ref{FiveOne} (\ref{FiveOneC}) \\ \hline  % 7 \\ \hline
        $\tt S^+ A^* A^*$ & $\VierFourD$ & Lemma \ref{VierFour} (\ref{VierFourD}) \\ \hline  % 10 \\ \hline
        $\tt S^+ A^+ A^*$ & $\SixOneA$ & Lemma \ref{SixOne} (\ref{SixOneA}) \\ \hline  % 13 \\ \hline
        $\tt S^+ A^+ A^+$ & $-\NSFrealA$ & Lemma \ref{NSFreal} (\ref{NSFrealA})  \\ \hline  % 14 \\ \hline
        $\tt S^+ A^+ A^-$ & $-\VierOneI$ & Lemma \ref{VierOne} (\ref{VierOneI}) \\ \hline  % 15 \\ \hline
        $\tt S^+ A^- A^*$ & $(\FiveTwoC)^{-1}$ & Lemma \ref{FiveTwo} (\ref{FiveTwoC}) \\ \hline  % 16 \\ \hline
        $\tt S^- A^* A^*$ & $-\VierFourD$ & Lemma \ref{VierFour} (\ref{VierFourD}) \\ \hline  % 19 \\ \hline
        $\tt S^- A^+ A^*$ & $\FiveOneD$ & Lemma \ref{FiveOne} (\ref{FiveOneD}) \\ \hline  % 22 \\ \hline
        $\tt S^- A^+ A^+$ & $(\VierOneJ)^{-1}$ & Lemma \ref{VierOne} (\ref{VierOneJ}) \\ \hline  % 23 \\ \hline
        $\tt S^- A^+ A^-$ & $-(\VierOneH)^{-1}$ & Lemma \ref{VierOne} (\ref{VierOneH}) \\ \hline  % 24 \\ \hline
        $\tt S^- A^- A^*$ & $-\FiveOneD$ & Lemma \ref{FiveOne} (\ref{FiveOneD}) \\ \hline  % 25 \\ \hline
        $\tt S^* A^+ N$ & $\VierTwoA$ & Lemma \ref{VierTwo} (\ref{VierTwoA})  \\ \hline  % 29 \\ \hline
        $\tt S^+ A^+ N$ & $\VierTwoC$ & Lemma \ref{VierTwo} (\ref{VierTwoC})  \\ \hline  % 32 \\ \hline
        $\tt S^- A^+ N$ & $\VierTwoB$ & Lemma \ref{VierTwo} (\ref{VierTwoB})  \\ \hline  % 35 \\ \hline
        $\tt S^* A^* S^*$ & $\VierFourI$ & Lemma \ref{VierFour} (\ref{VierFourI}) \\ \hline  % 37 \\ \hline
        $\tt S^* A^* S^+$ & $\VierFiveC$ & Lemma \ref{VierFive} (\ref{VierFiveC}) \\ \hline  % 38 \\ \hline
        $\tt S^* A^* S^-$ & $-\VierFiveC$ & Lemma \ref{VierFive} (\ref{VierFiveC}) \\ \hline  % 39 \\ \hline
        $\tt S^* A^+ S^*$ & $\FiveTwoG$ & Lemma \ref{FiveTwo} (\ref{FiveTwoG}) \\ \hline  % 40 \\ \hline
        $\tt S^* A^+ S^+$ & $\FiveTwoH$ & Lemma \ref{FiveTwo} (\ref{FiveTwoH}) \\ \hline  % 41 \\ \hline
        $\tt S^* A^+ S^-$ & $\FiveTwoI$ & Lemma \ref{FiveTwo} (\ref{FiveTwoI}) \\ \hline  % 42 \\ \hline
        $\tt S^* A^- S^*$ & $-\FiveTwoG$ & Lemma \ref{FiveTwo} (\ref{FiveTwoG}) \\ \hline  % 43 \\ \hline
        $\tt S^* A^- S^+$ & $-\FiveTwoH$ & Lemma \ref{FiveTwo} (\ref{FiveTwoH}) \\ \hline  % 44 \\ \hline
        $\tt S^* A^- S^-$ & $-\FiveTwoI$ & Lemma \ref{FiveTwo} (\ref{FiveTwoI}) \\ \hline  % 45 \\ \hline
        $\tt S^+ A^* S^*$ & $(\VierFiveC)^{-1}$ & Lemma \ref{VierFive} (\ref{VierFiveC}) \\ \hline  % 46 \\ \hline
        $\tt S^+ A^* S^-$ & $\VierFourJ$ & Lemma \ref{VierFour} (\ref{VierFourJ}) \\ \hline  % 48 \\ \hline
        $\tt S^+ A^+ S^*$ & $\SixOneB$ & Lemma \ref{SixOne} (\ref{SixOneB}) \\ \hline  % 49 \\ \hline
        $\tt S^+ A^+ S^+$ & $\NSFrealD$ & Lemma \ref{NSFreal} (\ref{NSFrealD})  \\ \hline  % 50 \\ \hline
        $\tt S^+ A^+ S^-$ & $\NSFrealC$ & Lemma \ref{NSFreal} (\ref{NSFrealC})  \\ \hline  % 51 \\ \hline
        $\tt S^+ A^- S^*$ & $(\FiveTwoI)^{-1}$ & Lemma \ref{FiveTwo} (\ref{FiveTwoI}) \\ \hline  % 52 \\ \hline
        $\tt S^+ A^- S^+$ & $(\FiveTwoL)^{-1}$ & Lemma \ref{FiveTwo} (\ref{FiveTwoL}) \\ \hline  % 53 \\ \hline
        $\tt S^+ A^- S^-$ & $-\NSFrealC$ & Lemma \ref{NSFreal} (\ref{NSFrealC})  \\ \hline  % 54 \\ \hline
        $\tt S^- A^* S^*$ & $-(\VierFiveC)^{-1}$ & Lemma \ref{VierFive} (\ref{VierFiveC}) \\ \hline  % 55 \\ \hline
        $\tt S^- A^* S^+$ & $-\VierFourJ$ & Lemma \ref{VierFour} (\ref{VierFourJ}) \\ \hline  % 56 \\ \hline
        $\tt S^- A^+ S^*$ & $\FiveTwoJ$ & Lemma \ref{FiveTwo} (\ref{FiveTwoJ}) \\ \hline  % 58 \\ \hline
        $\tt S^- A^+ S^+$ & $\FiveTwoK$ & Lemma \ref{FiveTwo} (\ref{FiveTwoK}) \\ \hline  % 59 \\ \hline
        $\tt S^- A^+ S^-$ & $\FiveTwoL$ & Lemma \ref{FiveTwo} (\ref{FiveTwoL}) \\ \hline  % 60 \\ \hline
        $\tt S^- A^- S^*$ & $-\FiveTwoJ$ & Lemma \ref{FiveTwo} (\ref{FiveTwoJ}) \\ \hline  % 61 \\ \hline
        $\tt S^- A^- S^+$ & $-\FiveTwoK$ & Lemma \ref{FiveTwo} (\ref{FiveTwoK}) \\ \hline  % 62 \\ \hline
        $\tt S^- A^- S^-$ & $-\FiveTwoL$ & Lemma \ref{FiveTwo} (\ref{FiveTwoL}) \\ \hline  % 63 \\ \hline
        $\tt S^+ N A^-$ & $\NSFrealB$ & Lemma \ref{NSFreal} (\ref{NSFrealB})  \\ \hline  % 69 \\ \hline
        $\tt S^- N A^+$ & $-\NSFrealB$ & Lemma \ref{NSFreal} (\ref{NSFrealB})  \\ \hline  % 71 \\ \hline
        $\tt S^+ N S^-$ & ~ & Lemma \ref{SSS etc} \\ \hline  % 81 \\ \hline
        $\tt S^- N S^+$ & $\NSFrealF$ & Lemma \ref{NSFreal} (\ref{NSFrealF}) \\ \hline  % 83 \\ \hline
        $\tt S^* S^* A^*$ & $\VierFourK$ & Lemma \ref{VierFour} (\ref{VierFourK}) \\ \hline  % 85 \\ \hline
        $\tt S^* S^* A^+$ & $\VierFiveD$ & Lemma \ref{VierFive} (\ref{VierFiveD}) \\ \hline  % 86 \\ \hline
        $\tt S^* S^* A^-$ & $-\VierFiveD$ & Lemma \ref{VierFive} (\ref{VierFiveD}) \\ \hline  % 87 \\ \hline
        $\tt S^* S^+ A^*$ & $\FiveOneE$  & Lemma \ref{FiveOne} (\ref{FiveOneE}) \\ \hline  % 88 \\ \hline
        $\tt S^* S^+ A^+$ & $\NSFrealD$ & Lemma \ref{NSFreal} (\ref{NSFrealD}) \\ \hline  % 89 \\ \hline
        $\tt S^* S^+ A^-$ & $\VierSixB$ & Lemma \ref{VierSix} (\ref{VierSixB}) \\ \hline  % 90 \\ \hline
        $\tt S^* S^- A^*$ & $-\FiveOneE$ & Lemma \ref{FiveOne} (\ref{FiveOneE}) \\ \hline  % 91 \\ \hline
        $\tt S^+ S^* A^*$ & $\VierFourN$ & Lemma \ref{VierFour} (\ref{VierFourN}) \\ \hline  % 94 \\ \hline
        $\tt S^+ S^+ A^*$ & $\SixOneC$ & Lemma \ref{SixOne} (\ref{SixOneC}) \\ \hline  % 97 \\ \hline
        $\tt S^+ S^+ A^-$ & $\FiveTwoE$ & Lemma \ref{FiveTwo} (\ref{FiveTwoE}) \\ \hline  % 99 \\ \hline
        $\tt S^+ S^- A^*$ & $\VierFourO$ & Lemma \ref{VierFour} (\ref{VierFourO}) \\ \hline  % 100 \\ \hline
        $\tt S^- S^* A^*$ & $-\VierFourN$ & Lemma \ref{VierFour} (\ref{VierFourN}) \\ \hline  % 103 \\ \hline
        $\tt S^- S^+ A^*$ & $\FiveOneF$ & Lemma \ref{FiveOne} (\ref{FiveOneF}) \\ \hline  % 106 \\ \hline
        $\tt S^- S^+ A^+$ & $-\FiveTwoE$ & Lemma \ref{FiveTwo} (\ref{FiveTwoE}) \\ \hline  % 107 \\ \hline
        $\tt S^- S^- A^*$ & $-\FiveOneF$ & Lemma \ref{FiveOne} (\ref{FiveOneF}) \\ \hline  % 109 \\ \hline
        $\tt S^* S^+ N$ &    & Lemma \ref{SSS etc} \\ \hline  % 113 \\ \hline
        $\tt S^* S^* S^*$ & $\VierSixA$ & Lemma \ref{VierSix} (\ref{VierSixA}) \\ \hline  % 121 \\ \hline
        $\tt S^* S^* S^+$ &    & Lemma \ref{SSS etc} \\ \hline  % 122 \\ \hline
        $\tt S^* S^* S^-$ &    & Lemma \ref{SSS etc} \\ \hline  % 123 \\ \hline
        $\tt S^* S^+ S^*$ & $\FiveTwoM$   & Lemma \ref{FiveTwo} (\ref{FiveTwoM}) \\ \hline  % 124 \\ \hline
        $\tt S^* S^+ S^+$ & $\FiveTwoN$   & Lemma \ref{FiveTwo} (\ref{FiveTwoN}) \\ \hline  % 125 \\ \hline
        $\tt S^* S^+ S^-$ & $\FiveTwoO$   & Lemma \ref{FiveTwo} (\ref{FiveTwoO}) \\ \hline  % 126 \\ \hline
        $\tt S^* S^- S^*$ & $-\FiveTwoM$   & Lemma \ref{FiveTwo} (\ref{FiveTwoM}) \\ \hline  % 127 \\ \hline
        $\tt S^* S^- S^+$ &    & Lemma \ref{SSS etc} \\ \hline  % 128 \\ \hline
        $\tt S^* S^- S^-$ &    & Lemma \ref{SSS etc} \\ \hline  % 129 \\ \hline
        $\tt S^+ S^* S^*$ & $-(\VierSixB)^{-1}$ & Lemma \ref{VierSix} (\ref{VierSixB}) \\ \hline  % 130 \\ \hline
        $\tt S^+ S^* S^-$ &    & Lemma \ref{SSS etc} \\ \hline  % 132 \\ \hline
        $\tt S^+ S^+ S^*$ & $\SixOneD$ & Lemma \ref{SixOne} (\ref{SixOneD}) \\ \hline  % 133 \\ \hline
        $\tt S^+ S^+ S^+$ &    & Lemma \ref{SSS etc} \\ \hline  % 134 \\ \hline
        $\tt S^+ S^+ S^-$ &    & Lemma \ref{SSS etc} \\ \hline  % 135 \\ \hline
        $\tt S^+ S^- S^*$ & $(\FiveTwoO)^{-1}$ & Lemma \ref{FiveTwo} (\ref{FiveTwoO}) \\ \hline  % 136 \\ \hline
        $\tt S^+ S^- S^+$ &    & Lemma \ref{SSS etc} \\ \hline  % 137 \\ \hline
        $\tt S^+ S^- S^-$ &    & Lemma \ref{SSS etc} \\ \hline  % 138 \\ \hline
        $\tt S^- S^* S^*$ & $(\FiveTwoO)^{-1}$   & Lemma \ref{FiveTwo} (\ref{FiveTwoO}) \\ \hline  % 139 \\ \hline
        $\tt S^- S^* S^+$ &    & Lemma \ref{SSS etc} \\ \hline  % 140 \\ \hline
        $\tt S^- S^+ S^*$ & $\FiveTwoP$ & Lemma \ref{FiveTwo} (\ref{FiveTwoP}) \\ \hline  % 142 \\ \hline
        $\tt S^- S^+ S^+$ &    & Lemma \ref{SSS etc} \\ \hline  % 143 \\ \hline
        $\tt S^- S^+ S^-$ &    & Lemma \ref{SSS etc} \\ \hline  % 144 \\ \hline
        $\tt S^- S^- S^*$ & $-\FiveTwoP$ & Lemma \ref{FiveTwo} (\ref{FiveTwoP}) \\ \hline  % 145 \\ \hline
        $\tt S^- S^- S^+$ &    & Lemma \ref{SSS etc} \\ \hline  % 146 \\ \hline
        $\tt S^- S^- S^-$ &    & Lemma \ref{SSS etc} \\ \hline  % 147 \\ \hline
\end{longtable}

Observe that the Inverse Theorem and Observation \ref{odd terms obs} are exploited in Table \ref{long table}.

\begin{rem}
\rm 
It follows immediately from the proof of Theorem \ref{main result for order 3} that the sequences that appear in Table \ref{long table} and  \cite[Table 2]{XMR-sepr}, together, constitute all the sequences of order $3$ that occur as a subsequence of the sepr-sequence of some Hermitian matrix
(i.e., the sequences of order $3$ that are \textit{not} forbidden).
\end{rem}

%%%%%%%%%%%%%%%%%%%%
%%%%%%%%%%%%%%%%%%%%
%%%%%%%%%%%%%%%%%%%%
%%%%%%%%%%%%%%%%%%%%
%%%%%%%%%%%%%%%%%%%%
%%%%%%%%%%%%%%%%%%%%
%%%%%%%%%%%%%%%%%%%%
%%%%%%%%%%%%%%%%%%%%
%%%%%%%%%%%%%%%%%%%%
\section{Real symmetric matrices}\label{s: real symmetric}
$\null$
\indent
This section focuses on real symmetric matrices,
and it provides a (complete) answer to Question \ref{question-real symmetric}, 
by establishing the analog results for real symmetric matrices of 
Theorem \ref{main result for order 2} and 
Theorem \ref{main result for order 3}.
%%%That is, we answer the following question: Which sequences of order 2 and order 3 do not occur as a subsequence of the sepr-sequence of any real symmetric matrix?

Theorem \ref{main result for order 2} asserts that
$\tt A^*N$, $\tt NA^*$, $\tt NS^*$ and $\tt S^*N$
are the only sequences from 
$\{\tt A^*,A^+,A^-,N,S^*,S^+,S^-\}$ 
of order 2 that are forbidden.
It follows from the proof of that theorem that
any sequence of order 2 that is not one of those four sequences occurs as a subsequence of the sepr-sequence of some real symmetric matrix. 
Thus, those four sequences are all the sequences of order $2$ that are forbidden over $\R$; 
that is encapsulated in the next theorem,
which provides one part of the answer to Question \ref{question-real symmetric}.

\begin{thm}\label{real symmetric main result for order 2}
A sequence from 
$\{\tt A^*,A^+,A^-,N,S^*,S^+,S^-\}$ of order $2$ does not occur as a subsequence of the sepr-sequence of any real symmetric matrix
if and only if 
it is one of the following sequences:
\[
\tt A^*N, \ NA^*, \ NS^* \mbox{ and \ } S^*N.
\]
\end{thm}

We just saw that replacing the word `Hermitian' in Theorem \ref{main result for order 2} with `real symmetric' results in a true statement (namely, Theorem \ref{real symmetric main result for order 2}).
Will the result be the same if we do that to Theorem \ref{main result for order 3}? 
We will now demonstrate that the answer is negative.
By Theorem \ref{main result for order 3}, 
the sequence $\tt NA^+A^*$, which is one of the sequences that appears in the first column of Table \ref{long table}, is not forbidden, implying that there exists a Hermitian matrix whose sepr-sequence contains $\tt NA^+A^*$ as a subsequence.  An example of such a matrix is provided in Table \ref{long table}; the example is a non-real matrix. The reason why a non-real matrix was provided is because there is no real Hermitian matrix whose sepr-sequence contains $\tt NA^+A^*$ as a subsequence
(i.e., $\tt NA^+A^*$ is forbidden over $\R$).

\begin{prop}\label{NA^+A^* result}
There is no real symmetric matrix whose sepr-sequence contains $\tt NA^+A^*$ as a subsequence.
\end{prop}

\begin{proof}
Let $B$ be a real symmetric matrix and 
$\sepr(B)=t_1t_2\cdots t_n$.
Suppose on the contrary that 
$t_{k}t_{k+1}t_{k+2} = \tt NA^+A^*$,
for some integer $k$ with $1 \leq k \leq n-2$.
Let $\epr(B) = \ell_1\ell_2 \cdots \ell_n$.
Thus, $\ell_k\ell_{k+1}\ell_{k+2} = \tt NAA$.
By the Basic Proposition, $k \geq 2$.
As no matrix has an sepr-sequence whose last term is 
$\tt A^*$, $n \geq k+3$. 
Thus, $k+1 \leq n-2$, implying that
$\ell_{k-1}\ell_{k}\ell_{k+1}$ is a 
subsequence of $\ell_1\ell_2 \cdots \ell_{n-2}$.
Then, as 
$\ell_{k-1}\ell_{k}\ell_{k+1} = \ell_{k-1}\tt NA$,
Theorem \ref{SNA} implies that $\ell_{k-1} \neq \tt S$.
By the $\tt NN$ Theorem, $\ell_{k-1} \neq \tt N$.
Thus, $\ell_{k-1} = \tt A$.
It follows, then, that $\ell_{k-1}\ell_k\ell_{k+1} = \tt ANA$.
Then, as  $2 \leq k \leq n-2$,
\cite[Theorem 5.6]{XMR-sepr} implies that,
for all $j \in \{1,2, \dots, n\}$,
$t_j \neq \tt A^*$,
contradicting the fact that $t_{k+2} = \tt A^*$.
\end{proof}

In the proof of Theorem \ref{main result for order 3},
for each order-3 sequence that is not forbidden, we provided a Hermitian matrix whose sepr-sequence contains it as a subsequence (by referring to either  Table \ref{long table} or \cite[Table 2]{XMR-sepr}). 
A remark about those matrices is in order.

\begin{rem}\label{nine sequences}
\normalfont
In the proof of Theorem \ref{main result for order 3},
for all but nine of the order-3 sequences that
are not forbidden,
the provided Hermitian matrix whose sepr-sequence contains it as a subsequence is real. 
The aforementioned nine sequences are the following:
$\tt NA^+A^*$, $\tt NA^+N$, 
$\tt NA^+S^*$, $\tt NA^+S^+$, $\tt NA^+S^-$,
$\tt NA^-N$,
$\tt NA^-S^*$, $\tt NA^-S^+$ and $\tt NA^-S^-$. 
Eight of those nine sequences appear  in Table \ref{long table}, while the remaining one, $\tt NA^-N$, appears in \cite[Table 2]{XMR-sepr}.
\end{rem}

Theorem \ref{real symmetric main result for order 2} and the next theorem, together, provide the answer to Question \ref{question-real symmetric}.

\begin{thm}\label{real symmetric main result for order 3}
Let $\sigma$ be a sequence from 
$\{\tt A^*,A^+,A^-,N,S^*,S^+,S^-\}$ of order $3$. 
Then $\sigma$ does not occur as a subsequence of the sepr-sequence of any real symmetric matrix
if and only if 
one of the following statements holds:
\begin{enumerate}
\item $\sigma$ does not occur as a subsequence of the sepr-sequence of any Hermitian matrix (see Theorem \ref{main result for order 3}). \label{Hermitian result}
\item $\sigma = \tt NA^+Z$ \ or \ $\sigma = \tt NA^-Z$, \ for some \ $\tt Z \in \{N,S^*,S^+,S^-\}$. \label{NAN and NAS real}
\item $\sigma = \tt NA^+A^*$. \label{NA+A*}
\end{enumerate}
\end{thm}

\begin{proof}
Suppose that $\sigma$ is forbidden over $\R$.
Suppose that (\ref{Hermitian result})  does not hold.
We will now prove that either (\ref{NAN and NAS real}) or (\ref{NA+A*}) holds.
As (\ref{Hermitian result})  does not hold, $\sigma$ is not forbidden.
Then, as $\sigma$ is forbidden over $\R$,
$\sigma$ is one of the nine sequences that are mentioned in Remark \ref{nine sequences}.
Thus, either (\ref{NAN and NAS real}) or (\ref{NA+A*}) holds, as desired.

We will now establish the other direction. 
If (\ref{Hermitian result})  holds,
then the desired conclusion follows from the fact that all real symmetric matrices are Hermitian.
If either (\ref{NAN and NAS real}) or (\ref{NA+A*}) holds, then the desired conclusion follows from 
Theorem \ref{NAN-NAS theorem} and
Proposition \ref{NA^+A^* result}.
\end{proof}

%%%The sequence $\sigma$ does not occur as a subsequence of the sepr-sequence of any Hermitian matrix.

The following corollary, which follows readily from Theorem \ref{main result for order 3}  and Theorem \ref{real symmetric main result for order 3},
implies that, of those sequences from 
$\{\tt A^*,A^+,A^-,N,S^*,S^+,S^-\}$ of order 3 that are not forbidden (see Theorem \ref{main result for order 3}),
the nine sequences mentioned in Remark \ref{nine sequences} are the only ones that are forbidden over $\R$.

\begin{cor}
Let $\sigma$ be a sequence from 
$\{\tt A^*,A^+,A^-,N,S^*,S^+,S^-\}$ of order $3$. 
Then $\sigma$ occurs as a subsequence of the sepr-sequence of some Hermitian matrix but not as a subsequence of the sepr-sequence of any real symmetric matrix if and only if one of the following statements holds:
\begin{enumerate}
\item $\sigma = \tt NA^+Z$ \ or \ $\sigma = \tt NA^-Z$, \ for some \ $\tt Z \in \{N,S^*,S^+,S^-\}$.
\item $\sigma = \tt NA^+A^*$.
\end{enumerate}
\end{cor}

%%%%%%%%%%%%%%%%%%%%
%%%%%%%%%%%%%%%%%%%%
%%%%%%%%%%%%%%%%%%%%
%%%%%%%%%%%%%%%%%%%%
%%%%%%%%%%%%%%%%%%%%
%%%%%%%%%%%%%%%%%%%%
%%%%%%%%%%%%%%%%%%%%
%%%%%%%%%%%%%%%%%%%%
%%%%%%%%%%%%%%%%%%%%
%\section{Conclusion}\label{s: conclusion}
%$\null$
%\indent
%A (complete) answer to Question \ref{question-Hermitian} was provided, by determining all the sequences from 
%$\{\tt A^*, A^+, A^-, N, S^*, S^+, S^-\}$ of orders $2$ and $3$ that do not occur as a subsequence of the sepr-sequence of any Hermitian matrix
%(see Theorem \ref{main result for order 2} and Theorem \ref{main result for order 3}).
%%%%%%
%%%%%%
%Moreover, we provided a (complete) answer to Question \ref{question-real symmetric}, the analogous question for real symmetric matrices, by determining all the sequences from 
%$\{\tt A^*, A^+, A^-, N, S^*, S^+, S^-\}$ of orders $2$ and $3$ that do not occur as a subsequence of the sepr-sequence of any real symmetric matrix
%(see Theorem \ref{real symmetric main result for order 2} and Theorem \ref{real symmetric main result for order 3}).

%%%%%%%%%%%%%%%%%%%%
%%%%%%%%%%%%%%%%%%%%
%%%%%%%%%%%%%%%%%%%%
%%%%%%%%%%%%%%%%%%%%
%%%%%%%%%%%%%%%%%%%%
%%%%%%%%%%%%%%%%%%%%
%%%%%%%%%%%%%%%%%%%%
%%%%%%%%%%%%%%%%%%%%
%%%%%%%%%%%%%%%%%%%%
\section*{Acknowledgments}
$\null$
\indent
Kamonchanok Saejeam's research was supported, in part, by a Summer Research Apprenticeship grant from the Dean of the Faculty's Office at Bates College.

%%%%%%%%%%%%%%%%%%%%
%%%%%%%%%%%%%%%%%%%%
%%%%%%%%%%%%%%%%%%%%
%%%%%%%%%%%%%%%%%%%%
%%%%%%%%%%%%%%%%%%%%
%%%%%%%%%%%%%%%%%%%%
%%%%%%%%%%%%%%%%%%%%
%%%%%%%%%%%%%%%%%%%%
%%%%%%%%%%%%%%%%%%%%


\begin{thebibliography}{99}

%\bibitem{BIRS13}
%W. Barrett, S. Butler, M. Catral, S. M. Fallat,
%H. T. Hall, L. Hogben, P. van den Driessche, M. Young.
%The principal rank characteristic sequence over various fields.
%\textit{Linear Algebra and its Applications}
%\textbf{459} (2014), 222--236.


\bibitem{pr}
R. A. Brualdi, L. Deaett, D. D. Olesky, P. van den Driessche.
The principal rank characteristic sequence of a real symmetric matrix.
\textit{Linear Algebra and its Applications}
\textbf{436} (2012), 2137--2155.


\bibitem{EPR}
S. Butler, M. Catral, S. M. Fallat, H. T. Hall, 
L. Hogben, P. van den Driessche, M. Young.
The enhanced principal rank characteristic sequence.
\textit{Linear Algebra and its Applications}
\textbf{498} (2016), 181--200.


\bibitem{EPR-Hermitian}
S. Butler, M. Catral, H. T. Hall, L. Hogben, X. Mart\'{i}nez-Rivera,
B. Shader, P. van den Driessche.
The enhanced principal rank characteristic sequence for Hermitian matrices.
\textit{Electronic Journal of Linear Algebra}
\textbf{32} (2017), 58--75.


\bibitem{HS}
O. Holtz, H. Schneider.
Open problems on GKK $\tau$-matrices.
\textit{Linear Algebra and its Applications}
\textbf{345} (2002), 263--267.


\bibitem{XMR-Classif}
X. Mart\'{i}nez-Rivera.
Classification of families of pr- and epr-sequences.
\textit{Linear and Multilinear Algebra}
\textbf{65} (2017), 1581--1599.


\bibitem{XMR-sepr}
X. Mart\'{i}nez-Rivera.
The signed enhanced principal rank characteristic sequence.
\textit{Linear and Multilinear Algebra}
\textbf{66} (2018), 1484--1503.


\bibitem{Zhang}
F. Zhang, 
\textit{Matrix Theory: Basic Results and Techniques}, 2nd ed., Springer, New York, 2011.

\end{thebibliography}
\end{document}